\documentclass[a4paper]{amsart}
\usepackage[utf8]{inputenc}
\usepackage[T1]{fontenc}
\usepackage{lmodern}
\usepackage{amssymb}
\usepackage[all]{xy}
\usepackage{nicefrac,mathtools,amsmath}
\usepackage[inline,shortlabels]{enumitem}
\usepackage{microtype}
\usepackage{graphicx}

\usepackage{xcolor}

\usepackage{tikz}
\usetikzlibrary{matrix,arrows}
\tikzset{cd/.style=matrix of math nodes,row sep=2em,column sep=2em, text height=1.5ex, text depth=0.5ex}
\tikzset{cdar/.style=->,auto}
\tikzset{dar/.style={double,double equal sign distance,-implies}}
\tikzset{mid/.style={anchor=mid}} 
\tikzset{narrowfill/.style={inner sep=0pt, fill=white}}

\newcommand*{\meascor}[5][1]{\draw[cdar] (#2) --
  node[inner sep=#1pt] {\(\scriptstyle #3\)}
  node[inner sep=#1pt,swap] {\(\scriptstyle #4\)} (#5);}

\setlist[enumerate,1]{label=\textup{(\arabic*)}}
\setlist[enumerate,2]{label=\textup{(\alph*)}}

\usepackage[pdftitle={A universal property for groupoid C*-algebras.  I},
pdfauthor={Alcides Buss, Rohit Holkar and Ralf Meyer},
pdfsubject={Mathematics}
]{hyperref}
\usepackage[lite]{amsrefs}
\renewcommand*{\MR}[1]{ \href{http://www.ams.org/mathscinet-getitem?mr=#1}{MR #1}}

\renewcommand{\PrintDOI}[1]{\href{http://dx.doi.org/#1}{DOI #1}}

\BibSpec{book}{%
  +{}  {\PrintPrimary}                {transition}
  +{,} { \textit}                     {title}
  +{.} { }                            {part}
  +{:} { \textit}                     {subtitle}
  +{,} { \PrintEdition}               {edition}
  +{}  { \PrintEditorsB}              {editor}
  +{,} { \PrintTranslatorsC}          {translator}
  +{,} { \PrintContributions}         {contribution}
  +{,} { }                            {series}
  +{,} { \voltext}                    {volume}
  +{,} { }                            {publisher}
  +{,} { }                            {organization}
  +{,} { }                            {address}
  +{,} { \PrintDateB}                 {date}
  +{,} { }                            {status}
  +{}  { \parenthesize}               {language}
  +{}  { \PrintTranslation}           {translation}
  +{;} { \PrintReprint}               {reprint}
  +{.} { }                            {note}
  +{.} {}                             {transition}
  +{} { \PrintDOI}                   {doi}
  +{}  {\SentenceSpace \PrintReviews} {review}
}

\numberwithin{equation}{section}

\theoremstyle{plain}
\newtheorem{theorem}[equation]{Theorem}
\newtheorem{lemma}[equation]{Lemma}
\newtheorem{proposition}[equation]{Proposition}

\newtheorem{corollary}[equation]{Corollary}

\theoremstyle{definition}
\newtheorem{definition}[equation]{Definition}

\theoremstyle{remark}
\newtheorem{remark}[equation]{Remark}

\newtheorem{example}[equation]{Example}

\newcommand*{\braket}[2]{\langle#1 {\mid} #2\rangle}

\newcommand*{\alb}{\hspace{0pt}} 

\newcommand*{\Mult}{\mathcal M}
\newcommand*{\U}{\mathcal U}
\newcommand*{\UM}{\mathcal U}
\newcommand*{\Lt}{\mathcal L}

\newcommand*{\s}{s} 
\newcommand*{\rg}{r}

\newcommand{\C}{\mathbb{C}}
\newcommand{\N}{\mathbb{N}}

\newcommand{\R}{\mathbb{R}}

\newcommand*{\Hilm}[1][E]{\mathcal #1}
\newcommand*{\F}{\Hilm[F]}
\newcommand*{\Hils}[1][H]{\mathcal #1}

\newcommand*{\univ}{\mathrm u}
\newcommand*{\diff}{\mathrm d}
\newcommand*{\dd}{\,\mathrm d} 
\newcommand*{\Cont}{\mathrm C}
\newcommand*{\Contc}{\mathrm{C_c}}
\newcommand*{\Contb}{\mathrm{C_b}}

\newcommand*{\Star}{$^*$\nobreakdash-}
\newcommand*{\nb}{\nobreakdash}
\newcommand*{\Cst}{\mathrm C^*}
\newcommand*{\cstar}{\texorpdfstring{\(\Cst\)\nobreakdash-\hspace{0pt}}{*-}}

\newcommand*{\congto}{\xrightarrow\sim}
\newcommand*{\defeq}{\mathrel{\vcentcolon=}}

\newcommand{\Bound}{\mathbb B}

\newcommand*{\abs}[1]{\lvert#1\rvert}
\newcommand*{\norm}[1]{\lVert#1\rVert}
\newcommand*{\conj}[1]{\overline{#1}}

\newcommand*{\Id}{\mathrm{id}}

\newcommand*{\into}{\hookrightarrow}
\newcommand*{\onto}{\twoheadrightarrow}

\DeclareMathOperator{\Hom}{Hom} 
\DeclareMathOperator{\supp}{supp}

\DeclareMathOperator{\pr}{pr}

\DeclareMathOperator{\Bis}{Bis}

\begin{document}
\title[A universal property for groupoid C*-algebras. I]
{A universal property for\\ groupoid C*-algebras. I}

\author{Alcides Buss}
\email{alcides@mtm.ufsc.br}
\address{Departamento de Matem\'atica\\
 Universidade Federal de Santa Catarina\\
 88.040-900 Florian\'opolis-SC\\
 Brazil}

\author{Rohit Holkar}
\email{rohit.d.holkar@gmail.com}
\curraddr[Rohit Holkar]{Indian Institute of Science Education and
  Research, Dr.~Homi Bhabha Road, NCL Colony, Pashan, Pune 411008,
  India}

\author{Ralf Meyer}
\email{rmeyer2@uni-goettingen.de}
\address{Mathematisches Institut\\
  Georg-August Universität Göttingen\\
  Bunsenstraße 3--5\\
  37073 Göttingen\\
  Germany}


\thanks{Supported by CNPq/CsF (Brazil).  Part of this work has been
  done while the second and third authors were visiting the
  Department of Mathematics at the Federal University of Santa
  Catarina in Florianópolis, and they thank this institution for the
  hospitality.}

\begin{abstract}
  We describe representations of groupoid \cstar{}algebras on
  Hilbert modules over arbitrary \cstar{}algebras by a universal
  property.  For Hilbert space representations, our universal
  property is equivalent to Renault's Integration--Disintegration
  Theorem.  For a locally compact group, it is related to the
  automatic continuity of measurable group representations.  It
  implies known descriptions of groupoid \cstar{}algebras as crossed
  products for \'etale groupoids and transformation groupoids of group
  actions on spaces.
\end{abstract}
\maketitle

\section{Introduction}
\label{sec:intro}

The \cstar{}algebra of a locally compact group~\(G\) may be
characterised uniquely up to isomorphism by a universal property:
there is a natural bijection between nondegenerate
\Star{}homomorphisms \(\Cst(G)\to \Mult(D)\) --~briefly called
morphisms \(\Cst(G)\to D\)~-- and strictly continuous group
homomorphisms \(G\to \UM(D)\), where~\(D\) is a
\cstar{}algebra and~\(\UM(D)\) its group of unitary multipliers.
That this universal property characterises \(\Cst(G)\) is shown in
\cite{Williams:crossed-products}*{Theorem 2.61}, even in the more
general case of crossed products; Williams attributes this theorem
to Raeburn.

Now let~\(G\) be a locally compact, Hausdorff groupoid with a Haar
system~\(\alpha\).  Renault's Integration and Disintegration Theorems describe
the Hilbert space representations of the groupoid
\cstar{}algebra~\(\Cst(G,\alpha)\)
(see~\cite{Renault:Representations}).
The category of Hilbert space representations and intertwining
operators is not enough, however, to
determine a \(\Cst\)\nb-algebra uniquely up to isomorphism, compare
\cite{Rieffel:Morita}*{Example 1.4}.  This complicates
many arguments about groupoid \cstar{}algebras because the
dense \Star{}subalgebra \(\Contc(G)\) and other details of the
definition of~\(\Cst(G,\alpha)\) reappear in every argument.
This includes such technical matters as the automatic
boundedness of any \Star{}representation of \(\Contc(G)\) in the
\(I\)\nb-norm.

We are going to describe the representations of~\(\Cst(G,\alpha)\)
on Hilbert \emph{modules} over arbitrary \cstar{}algebras by a
universal property.  This universal property
determines~\(\Cst(G,\alpha)\) uniquely up to a canonical isomorphism
and should, therefore, simplify many arguments with groupoid
\cstar{}algebras.  We also prove the automatic boundedness
of densely defined \Star{}representations of \(\Contc(G)\), even
if~\(G\) is not second countable (Corollary~\ref{cor:prerep}).
This allows us to get rid of the second countability assumption in
the main result of~\cite{Muhly-Renault-Williams:Equivalence},
which says that Morita equivalent groupoids have Morita equivalent
\(\Cst\)\nb-algebras.

In particular, this result implies that the groupoid
\cstar{}algebras for different Haar systems on the same groupoid
are canonically Morita--Rieffel equivalent.  It is unclear,
however, whether they are isomorphic, and there certainly is no
\emph{natural} isomorphism between them.  Hence a universal
property for groupoid \(\Cst\)\nb-algebras must contain the Haar
system.  This entails some complications.  To show how our
universal property can be used, we apply it to two special cases,
namely, étale groupoids and transformation groupoids of group
actions.  We describe their representations and thus their groupoid
\cstar{}algebras.  This implies that the groupoid \cstar{}algebra of
an étale groupoid or a transformation groupoid for a group action is
a crossed product for an inverse semigroup action or a group action,
respectively.  This description comes with a universal property that
describes representations on Hilbert modules as well as Hilbert
spaces.
Hilbert modules are powerful objects, and some proofs of
disintegration theorems already use them.  For instance,
\cite{Androulidakis-Skandalis:Holonomy} proves a disintegration
theorem for Hilbert space representations of holonomy groupoids of
singular foliations.  Our techniques are, however, quite different
from those in~\cite{Androulidakis-Skandalis:Holonomy}.

Our universal property uses the commutative \cstar{}algebras of
functions on the spaces of objects, arrows, and composable pairs of
arrows in~\(G\).  Therefore, as it stands, it only works for
Hausdorff groupoids.  The non-Hausdorff case may be treated by
desingularising a non-Hausdorff, locally compact groupoid to a
Hausdorff, locally compact bigroupoid.  There is also a
variant where we add Fell bundles, even non-saturated ones.  The
most general version of the universal property applies to
non-saturated Fell bundles over bigroupoids, which we view as
partial actions of bigroupoids by Hilbert bimodules (partial
Morita--Rieffel equivalences).  Since both Fell bundles and
bigroupoids create further technical complications, we discuss them
only later, in sequels to this article.

When we combine our universal property with the representation theory
of commutative \cstar{}algebras
on separable Hilbert spaces, the resulting description of
representations of groupoid \cstar{}algebras
on separable Hilbert spaces is equivalent to Renault's
Integration--Disintegration Theorem.  Besides the Haar system on the
groupoid, our universal property does not involve any measure theory
because this would fail for representations on Hilbert
\emph{modules}: direct integral decompositions need not
exist in this case, see Remark~\ref{rem:no_disintegration_module}.
In fact, our universal property works for arbitrary (non-separable)
groupoid \cstar{}algebras and only involves rather soft analysis.
This is compensated by appropriate algebraic structures.  For a
locally compact group~\(G\), our universal property for Hilbert
space representations gives Haar-measurable weak representations, that
is, Haar-measurable maps \(g\mapsto U_g\) from~\(G\) to the unitary
group such that \(U_g U_h = U_{g h}\) holds for almost all \((g,h)\in
G^2\) with respect to the Haar measure.  Together with the usual
universal property for group \cstar{}algebras, this shows that any
Haar-measurable weak group representation is equal almost everywhere
to a continuous group representation (see
Corollary~\ref{cor:automatic-continuity}).
Similar automatic continuity results for group representations go
back to Stefan Banach and André
Weil (see \cite{Rosendal:Automatic}).

Throughout this article, we let~\(G\) be a locally compact,
Hausdorff groupoid with a Haar system, which we denote
by~\(\alpha\).  Let \(G^0\), \(G^1\) and~\(G^2\) be its spaces
of objects, arrows and composable pairs of arrows, and let
\(\rg,\s\colon G^1\rightrightarrows G^0\) be its range and source
maps.  We recall in Section~\ref{sec:families_correspondences} how
to construct \cstar{}correspondences between commutative
\cstar{}algebras such as \(\Cont_0(G^i)\) for \(i=0,1,2\) from
topological correspondences between the underlying spaces.  We
construct some \cstar{}correspondences of this type from
families of measures along canonical maps
\(G^2 \mathrel{\substack{\textstyle\rightarrow\\[-0.6ex]
    \textstyle\rightarrow \\[-0.6ex]
    \textstyle\rightarrow}}
G^1 \mathrel{\substack{\textstyle\rightarrow\\[-0.6ex]
    \textstyle\rightarrow}}
G^0\).
Using these \cstar{}correspondences, we formulate our universal
property in Section~\ref{sec:universal_groupoid_Hausdorff}.  We
illustrate it by the example of the regular representation and
relate it to the Integration and Disintegration Theorems of
Renault~\cite{Renault:Representations}.  For the universal property,
we define ``representations'' of~\((G,\alpha)\) on Hilbert modules.
Our main theorem asserts that these representations are equivalent
to representations of the groupoid \cstar{}algebra.  We describe
how to integrate and disintegrate representations in Sections
\ref{sec:integrate} and~\ref{sec:disintegrate}, and we show that
both constructions are inverse to each other in
Section~\ref{sec:back_and_forth}.  This section finishes the proof
of the universal property.  Section~\ref{sec:trafo_group_etale}
specialises to transformation groups and \'etale groupoids.

\section{Continuous families of measures and topological correspondences}
\label{sec:families_correspondences}

Our universal property is based on canonical
\cstar{}correspondences
between the commutative \cstar{}algebras
\(\Cont_0(G^i)\)
for \(i=0,1,2\).
We are going to describe a general procedure to construct
\cstar{}correspondences between commutative \cstar{}algebras.  The
\cstar{}correspondences
between \(\Cont_0(G^i)\) that we need are all of this form.

A \cstar{}correspondence from a \cstar{}algebra~\(A\)
to another \cstar{}algebra~\(D\)
consists of a (right) Hilbert \(D\)\nb-module~\(\Hilm[F]\)
with a nondegenerate \Star{}\hspace{0pt}homomorphism~\(\varphi\) from~\(A\)
to~\(\Bound(\Hilm[F])\),
the \cstar{}algebra of adjointable operators on~\(\Hilm[F]\).
We view a \cstar{}correspondence from~\(A\)
to~\(D\)
as an arrow \(A\to D\)
and usually write \(A\xrightarrow{\Hilm[F]} D\).
We also view~\(\varphi\)
as a \emph{representation} of~\(A\)
on~\(\Hilm[F]\).
Two \cstar{}correspondences
\(\Hilm[F_1]\)
and~\(\Hilm[F_2]\)
from~\(A\)
to~\(D\)
are \emph{isomorphic} if there is a unitary bimodule map
\(U\colon \Hilm[F_1] \congto \Hilm[F_2]\).

We write~\(\otimes\)
for suitably completed tensor products of \cstar{}correspondences,
and~\(\odot\)
for the tensor product of vector spaces without any completion.
In particular, the composite of two \cstar{}correspondences
\(A\xrightarrow{\Hilm} B\) and \(B\xrightarrow{\Hilm[F]} D\) is their
(balanced) tensor product~\(\Hilm\otimes_B\Hilm[F]\)
(see~\cite{Lance:Hilbert_modules}*{Chapter~4}). This is
the completion of the algebraic (balanced) tensor product
\(\Hilm\odot_B\Hilm[F]\) with respect to the \(D\)\nb-valued inner product
\[
\braket{\xi_1\otimes\eta_1}{\xi_2\otimes \eta_2}\defeq
\braket{\eta_1}{\varphi(\braket{\xi_1}{\xi_2})\eta_2},
\]
where \(\varphi\colon B\to \Bound(\Hilm[F])\) is the underlying
homomorphism that gives the left \(B\)\nb-module structure of~\(\Hilm[F]\);
the notation \(\Hilm\otimes_\varphi\Hilm[F]\) is used instead of
\(\Hilm\otimes_B\Hilm[F]\) to highlight~\(\varphi\).

Let \(X\)
and~\(Y\)
be locally compact, Hausdorff spaces and let \(f\colon X\to Y\)
be a continuous map with a continuous family~\(\lambda\)
of measures~\(\lambda_y\)
along the fibres~\(f^{-1}(y)\)
of~\(f\)
(such families are called \emph{\(f\)\nb-systems} in
\cite{Renault:Representations}*{Section~1}).  Thus
each~\(\lambda_y\)
is a positive Radon measure on~\(X\)
with \(\supp(\lambda_y)\subseteq f^{-1}(y)\).
The continuity of~\(\lambda\)
means that the integration map
\[
\lambda\colon \Contc(X) \to \Contc(Y),\qquad
\lambda(\varphi)(y) = \int_X \varphi(x)\dd\lambda_y(x),
\]
takes values in~\(\Contc(Y)\).

\begin{definition}
  \label{def:corr_from_measures}
  We equip
  \(\Contc(X)\) with the \(\Cont_0(X)\)-\(\Cont_0(Y)\)\nb-bimodule
  structure
  \[
  (\varphi_1\cdot \varphi_2\cdot \varphi_3)(x) \defeq
  \varphi_1(x)\varphi_2(x)\varphi_3(f(x))
  \]
  for \(\varphi_1\in\Cont_0(X)\),
  \(\varphi_2\in\Contc(X)\),
  \(\varphi_3\in\Cont_0(Y)\)
  and with the \(\Cont_0(Y)\)-valued
  inner product
  \(\braket{\xi}{\eta} \defeq \lambda \bigl(\conj{\xi}\cdot
  \eta\bigr)\), that is,
  \[
  \braket{\xi}{\eta}(y)
  \defeq \int_{f^{-1}(y)} \conj{\xi(x)}\cdot \eta(x)
  \dd\lambda_y(x).
  \]
  Then \(\Contc(X)\)
  is a pre-\alb{}Hilbert \(\Cont_0(Y)\)-module
  with a nondegenerate representation of~\(\Cont_0(X)\)
  by adjointable operators.  Some nonzero \(\xi\in \Contc(X)\)
  might have \(\braket{\xi}{\xi}=0\)
  unless we assume~\(\lambda_y\)
  to have full support.  Let \(\Lt^2(X,f,\lambda)\)
  be the Hausdorff completion of~\(\Contc(X)\)
  for this inner product, which is a \cstar{}correspondence
  from~\(\Cont_0(X)\)
  to~\(\Cont_0(Y)\).
  In diagrams, we often briefly denote this \cstar{}correspondence
  as
  \[
  \Cont_0(X) \xrightarrow[\lambda]{f} \Cont_0(Y).
  \]
\end{definition}


Any Hilbert module over a commutative
\cstar{}algebra~\(\Cont_0(Y)\) is isomorphic to the section space
of a continuous field of Hilbert spaces over~\(Y\), see
\cite{Weaver:book}*{Section~9.1}.  This is an equivalence of
categories, that is, there is a functorial bijection between
unitary operators between two Hilbert \(\Cont_0(Y)\)-modules and
continuous families of unitary operators between the corresponding
continuous fields of Hilbert spaces over~\(Y\).  The continuous
field of Hilbert spaces associated to \(\Lt^2(X,f,\lambda)\) has
the fibre \(\Lt^2(f^{-1}(y),\lambda_y)\) at \(y\in Y\), and its
\(\Cont_0\)\nb-sections are generated by~\(\Contc(X)\); here
\(\xi\in \Contc(X)\) is identified with the section \(y\mapsto
\xi|_{f^{-1}(y)}\).

Let \(f\colon X\to Y\) and \(g\colon Y\to Z\) be continuous maps
with continuous families of measures \(\lambda\) and~\(\mu\),
respectively.  Then the composite integration map
\begin{equation}
  \mu\circ\lambda\colon \Contc(X)\to\Contc(Y)\to\Contc(Z),\qquad
  (\mu\circ\lambda)(\varphi)(z)
  = \int_Y \int_X \varphi(x)\dd\lambda_y(x)\dd\mu_z(y),
\end{equation}
describes a continuous family of measures \(\mu\circ\lambda\) along
\(g\circ f\).

\begin{lemma}
  \label{lem:compose_measure_family}
  The map
  \[
  \gamma\colon \Contc(X)\odot \Contc(Y) \to \Contc(X),\qquad
  \gamma(\varphi\otimes\psi)(x) \defeq \varphi(x)\cdot\psi(f(x)),
  \]
  extends uniquely to an isomorphism of
  \(\Cont_0(X)\)-\(\Cont_0(Z)\)-correspondences
  \[
  \bar\gamma\colon \Lt^2(X,f,\lambda) \otimes_{\Cont_0(Y)} \Lt^2(Y,g,\mu)
  \congto \Lt^2(X,g\circ f,\mu\circ \lambda).
  \]
\end{lemma}

\begin{proof}
  The inner products defining \(\Lt^2(X,f,\lambda)
  \otimes_{\Cont_0(Y)} \Lt^2(Y,g,\mu)\) and \(\Lt^2(X,g\circ
  f,\mu\circ \lambda)\) are preserved by~\(\gamma\) because
  \begin{align*}
    \braket{\gamma(\varphi_1\otimes\psi_1)}{\gamma(\varphi_2\otimes\psi_2)}(z)
    &= \int_{g^{-1}(z)} \int_{f^{-1}(y)} \conj{\gamma(\varphi_1\otimes\psi_1)(x)}
    \gamma(\varphi_2\otimes\psi_2)(x) \dd\lambda_y(x)\dd\mu_z(y)
    \\&= \int_{g^{-1}(z)} \int_{f^{-1}(y)} \conj{\varphi_1(x)\psi_1(y)}
    \varphi_2(x)\psi_2(y) \dd\lambda_y(x)\dd\mu_z(y)
    \\&= \braket{\psi_1}{\braket{\varphi_1}{\varphi_2}\cdot\psi_2}(z)
  \end{align*}
  for all \(z\in Z\).
  Hence~\(\gamma\)
  extends to an isometry
  \[
  \bar\gamma\colon \Lt^2(X,f,\lambda) \otimes_{\Cont_0(Y)}
  \Lt^2(Y,g,\mu)
  \to \Lt^2(X,g\circ f,\mu\circ \lambda).
  \]
  This is clearly a bimodule map.  It is surjective (and hence an
  isomorphism of correspondences)
  because~\(\gamma\colon \Contc(X)\odot \Contc(Y) \to \Contc(X)\)
  is already surjective: any \(\varphi'\in \Contc(X)\)
  may be decomposed as \(\varphi'(x)=\varphi(x)\psi(f(x))\)
  by taking \(\psi\in \Contc(Y)\)
  with \(\psi(y)=1\)
  for \(y\in f(\supp(\varphi'))\) and \(\varphi\defeq \varphi'\).
\end{proof}

The compositions for measure families and \cstar{}correspondences
are compatible by Lemma~\ref{lem:compose_measure_family}.  In our
brief notation, this means that the following diagram of
\cstar{}correspondences commutes up to the canonical
isomorphism~\(\gamma\):
\begin{equation}
  \label{eq:compose_measure_family}
  \begin{tikzpicture}[baseline=(current bounding box.west)]
    \matrix (m) [cd,column sep=4em, row sep=2.5em] {
      \Cont_0(X)&
      \Cont_0(Y)\\
      &\Cont_0(Z)\\
    };
    \meascor{m-1-1}{f}{\lambda}{m-1-2};
    \meascor{m-1-2}{g}{\mu}{m-2-2};
    \meascor{m-1-1}{g\circ f}{\mu\circ\lambda}{m-2-2};
  \end{tikzpicture}
\end{equation}

\begin{definition}
  Let \(f\colon X\to Y\) (``forward'') be a continuous map with a
  continuous family of measures~\(\lambda\) and let \(b\colon X\to Z\)
  (``backward'') be a continuous map.  We define a
  \(\Cont_0(Z)\)-module structure on~\(\Contc(X)\) by \((\varphi\cdot
  \psi)(x) \defeq \varphi(b(x))\cdot \psi(x)\) for
  \(\varphi\in\Cont_0(Z)\), \(\psi\in\Contc(X)\).  This extends to a
  representation of~\(\Cont_0(Z)\) on the Hilbert
  \(\Cont_0(Y)\)\nb-module \(\Lt^2(X,f,\lambda)\), turning it into
  a \cstar{}correspondence from~\(\Cont_0(Z)\)
  to~\(\Cont_0(Y)\).  We denote it by \(b^*\Lt^2(X,f,\lambda)\).
  We call a pair of maps \(Z \xleftarrow{b} X
  \xrightarrow{f} Y\) with a continuous family of measures~\(\lambda\)
  along~\(f\) a \emph{topological correspondence} from~\(Z\)
  to~\(Y\).
\end{definition}

In particular, the \(\Cont_0(X)\)-\(\Cont_0(Y)\)\nb-correspondence
\(\Lt^2(X,f,\lambda)\) built in
Definition~\ref{def:corr_from_measures} from a continuous family of
measures~\(\lambda\) along a continuous map \(f\colon X\to Y\) is
associated to the topological correspondence
\[
X\xleftarrow{\Id_X} X \xrightarrow[\lambda]{f} Y.
\]

Topological correspondences are a mild generalisation of the
topological quivers introduced by Muhly and
Tomforde~\cite{Muhly-Tomforde:Quivers}: a topological quiver is a
topological correspondence with the same source and target space.
Basic results about topological quivers such as
\cite{Muhly-Tomforde:Quivers}*{Lemmas 6.1--4} have obvious
generalisations to topological correspondences.  We also get the
notion of a topological correspondence if we specialise the
topological correspondences between locally compact, Hausdorff
groupoids introduced in~\cite{Holkar:Thesis} to locally compact
spaces.  

Topological correspondences may be composed by a fibre product
construction, and this composition and the interior tensor product of
\cstar{}correspondences
are compatible up to a canonical isomorphism, see
\cite{Muhly-Tomforde:Quivers}*{Lemmas 6.1--4} or~\cite{Holkar:Thesis}.
We give more details.  Let \(X\),
\(Y\)
and~\(Z\)
be locally compact spaces.  Let \((V,b_V,f_V,\lambda)\)
and \((W,b_W,f_W,\mu)\)
be topological correspondences from~\(X\)
to~\(Y\)
and from~\(Y\)
to~\(Z\),
respectively.  Their composite topological correspondence is
the fibre product
\[
V\times_{f_V,Y,b_W}W\defeq \{(v,w)\in V\times W: f_V(v)=b_W(w)\}
\]
with the maps \(b\defeq b_V\circ \pr_1\)
and \(f\defeq f_W\circ\pr_2\),
respectively, where~\(\pr_i\)
is the projection from \(V\times_{f_V,Y,b_W}W\)
to the \(i\)th
factor; the family of measures \(\lambda\times\mu\)
on \(V\times_{f_V,Y,b_W}W\) is defined by
\[
\int_{V\times_{f_V,Y,b_W}W} \varphi \dd(\lambda\times\mu)_z
\defeq \int_{f_W^{-1}(z)}\int_{f_V^{-1}(b_W(w))}
\varphi(v,w) \dd\lambda_{b_W(w)}(v) \dd\mu_z(w)
\]
for \(\varphi\in \Contc(V\times_{f_V,Y,b_W}W)\) and \(z\in Z\).

\begin{proposition}
  \label{prop:functoriality-of-corr}
  The canonical map
  \[
  \gamma\colon\Contc(V)\odot\Contc(W)\to\Contc(V\times_{f_V,Y,b_W}W),
  \qquad
  \gamma(\varphi\otimes \psi)(v,w)\defeq \varphi(v)\psi(w),
  \]
  extends to an isomorphism
  \[
  b_V^* \Lt^2(V,f_V,\lambda)\otimes_{\Cont_0(Y)}
  b_W^* \Lt^2(W,f_W,\mu) \congto
  b^* \Lt^2(V\times_{f_V,Y,b_W}W,f,\lambda\times\mu)
  \]
  of \cstar{}correspondences from \(\Cont_0(X)\) to \(\Cont_0(Z)\).
\end{proposition}

\begin{proof}
  A direct computation shows that~\(\gamma\)
  is a bimodule map.  To see that~\(\gamma\)
  preserves the inner products, take
  \(\varphi_1,\varphi_2\in \Contc(V)\)
  and \(\psi_1,\psi_2\in \Contc(W)\).  Then
  \begin{align*}
    \braket{\varphi_1\otimes\psi_1}{\varphi_2\otimes\psi_2}(z)
    &= \braket{\psi_1}{\braket{\varphi_1}{\varphi_2}\cdot\psi_2}(z)
    \\&= \int_{f_W^{-1}(z)}\int_{f_V^{-1}(b_W(w))}
    \conj{\psi_1(w)}\conj{\varphi_1(v)}
    \varphi_2(v)\psi_2(w)\dd\lambda_{b_W(w)}(v)\dd\mu_z(w)
    \\&=\braket{\gamma(\varphi_1\otimes\psi_1)}
    {\gamma(\varphi_2\otimes\psi_2)}(z).
  \end{align*}
  Thus~\(\gamma\)
  is an isometric bimodule map.  The subspace of \(\Contc(V\times W)\)
  spanned by functions of the form \((v,w)\mapsto \varphi(v)\psi(w)\)
  is dense in the inductive limit topology.  Hence restrictions of
  such functions to \(V\times_{f_V,Y,b_W}W\)
  are linearly dense in \(\Contc(V\times_{f_V,Y,b_W}W)\),
  which is dense in
  \(\Lt^2(V\times_{f_V,Y,b_W}W,f,\lambda\times\mu)\).
  Thus~\(\gamma\) is surjective.
\end{proof}

An \emph{isomorphism} between two topological correspondences
\[
X\xleftarrow{b_i} Z_i\xrightarrow[\lambda_i]{f_i} Y,\qquad
i=1,2,
\]
is a homeomorphism \(\Phi\colon Z_1\congto Z_2\) that
satisfies \(f_2\circ\Phi=f_1\), \(b_2\circ\Phi=b_1\) with a
continuous function \(\delta\colon Z_2 \to \R_{>0}\) such that
\(\lambda_2 = \delta\cdot \Phi_*(\lambda_1)\), that is,
\(\lambda_{2,y} = \delta\cdot \Phi_*(\lambda_1)_y\) for all \(y\in
Y\).  We call~\(\delta\) an \emph{equivalence} from
\(\Phi_*(\lambda_1)\) to \(\lambda_2\).  The restriction
of~\(\delta\) to a fibre \(f_2^{-1}(y)\subseteq Z_2\) is a
Radon--Nikodym derivative for the measures \(\lambda_{2,y}\)
and~\(\Phi_*(\lambda_1)_y\).  This is unique up to equality
almost everywhere.  Therefore, if \(\lambda_1\) and~\(\lambda_2\)
have full support and a function~\(\delta\) as above exists, then
it is unique.  An isomorphism as above induces a unitary
\(\Cont_0(X)\)-\(\Cont_0(Y)\)-bimodule isomorphism
\[
\Phi_*\colon \Contc(Z_1)\to \Contc(Z_2),\qquad
\Phi_*(g)(z) \defeq
g\circ \Phi^{-1}(z)\, \delta(z)^{-1/2}.
\]
It extends to an isomorphism of \cstar{}correspondences
\[
\Phi_*\colon b_1^* \Lt^2(Y,f_1,\lambda_1)\congto
b_2^* \Lt^2(Y,f_2,\lambda_2).
\]

\begin{remark}
  We may weaken the continuity assumptions on \(\Phi^{\pm1}\)
  and~\(\delta\).  For instance, if~\(Y\) is just a point, then it
  suffices to assume~\(\Phi^{\pm1}\) to be measurable.  Then we
  take~\(\delta\) to be the Radon--Nikodym derivative as above,
  which is automatically measurable.  In general, we need
  \(\Phi^{\pm1}\) and~\(\delta\) to be measurable in the fibre
  directions and continuous along the base~\(Y\).  We do not try
  to make this precise here because we shall only use continuous
  isomorphisms as defined above.
\end{remark}

\section{The universal property for groupoid \texorpdfstring{$\Cst$}{C*}-algebras}
\label{sec:universal_groupoid_Hausdorff}

Let~\(\alpha\) be a left-invariant Haar system on~\(G\).  This is a
continuous family of measures with full support along the fibres
\(G^x \defeq \{g\in G\mid \rg(g)=x\}\) of the range map \(\rg\colon
G^1\onto G^0\).  The right-invariant Haar system~\(\tilde{\alpha}\)
corresponding to~\(\alpha\) is the push-forward
of~\(\alpha\) along the inversion map.  So it is a continuous family
of measures along the fibres \(G_x \defeq \{g\in G\mid \s(g)=x\}\)
of the source map \(\s\colon G^1\onto G^0\):
\begin{align*}
  (\alpha f)(x) &= \int_{G^x} f(g)\dd\alpha^x(g),\\
  (\tilde{\alpha} f)(x) &= \int_{G_x} f(g)\dd\tilde{\alpha}_x(g)
  = \int_{G^x} f(g^{-1})\dd\alpha^x(g).
\end{align*}

Besides the range and source maps \(\rg,\s\colon G^1\onto G^0\), we
shall also use the three maps \(d_0,d_1,d_2\colon G^2\onto G^1\)
defined by
\[
d_0(g,h) = h,\qquad
d_1(g,h) = g\cdot h,\qquad
d_2(g,h) = g
\]
for \(g,h\in G^1\) with \(\s(g)=\rg(h)\).  The composite maps are
\begin{equation}
  \label{eq:vertex_maps}
  v_0 = \rg\circ d_1 = \rg\circ d_2,\qquad
  v_1 = \rg\circ d_0 = \s\circ d_2,\qquad
  v_2 = \s\circ d_0 = \s\circ d_1.
\end{equation}
These maps are illustrated in Figure~\ref{fig:triangle_arrows}.
\begin{figure}[htbp]
  \begin{tikzpicture}[scale=1.8]
    \node (v1) at (90:1) {\(v_1\)};
    \node (v0) at (210:1) {\(v_0\)};
    \node (v2) at (330:1) {\(v_2\)};
    \draw[cdar] (v1) -- node[swap] {\(d_2(g,h)=g\)} (v0);
    \draw[cdar] (v2) -- node {\(d_1(g,h)=gh\)} (v0);
    \draw[cdar] (v2) -- node[swap] {\(d_0(g,h)=h\)} (v1);
  \end{tikzpicture}
  \caption{A pair \((g,h)\in G^2\) of composable arrows generates a
    commutative triangle of arrows in~\(G\).  We number
    the edges so that the one opposite the vertex~\(v_i\) is
    denoted by~\(d_i\) for \(i=0,1,2\).}
  \label{fig:triangle_arrows}
\end{figure}
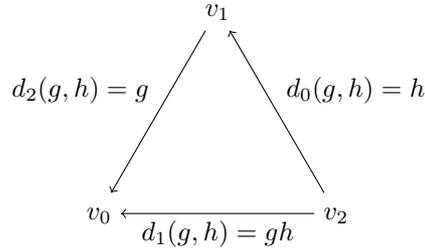

Our continuous families of measures \(\tilde\alpha\) and~\(\alpha\)
along \(\s\) and~\(\rg\) induce continuous families \(\lambda_0\),
\(\lambda_1\) and~\(\lambda_2\) along the maps \(d_0\), \(d_1\)
and~\(d_2\), which we describe through their integration maps
\(\Contc(G^2)\to\Contc(G^1)\):
\begin{align}
  \label{eq:lambda_definition0}
  (\lambda_0 f)(h) &=
  \int_{G_{\rg(h)}} f(g,h) \dd\tilde{\alpha}_{\rg(h)}(g),\\
  \label{eq:lambda_definition1}
  (\lambda_1 f)(k) &=
  \int_{G^{\rg(k)}} f(g,g^{-1}k) \dd\alpha^{\rg(k)}(g)
  = \int_{G_{\s(k)}} f(kh^{-1},h) \dd\tilde{\alpha}_{\s(k)}(h),\\
  \label{eq:lambda_definition2}
  (\lambda_2 f)(g) &=
  \int_{G^{\s(g)}} f(g,h) \dd\alpha^{\s(g)}(h).
\end{align}
Equation~\ref{eq:lambda_definition1} uses the
substitution \(h=g^{-1}k\), which transforms the measures as
asserted because~\(\alpha\) is left invariant and \(\tilde{\alpha}\)
is obtained from~\(\alpha\) by the substitution \(g\mapsto g^{-1}\).

Each map in~\eqref{eq:vertex_maps} comes with a fixed
continuous family of measures.  As the following computations
show, each identity in~\eqref{eq:vertex_maps} corresponds
to an identity of measure families or, equivalently, integration
maps:
\begin{align}
  \label{eq:vertex_maps_integrals0}
  (\alpha\circ \lambda_1) f(x)
  &= \int_{G^x} \int_{G^x} f(g,g^{-1} k)
  \dd\alpha^x(g)\dd\alpha^x(k)\\
  \notag &= \int_{G^x} \int_{G^x} f(g,h)
  \dd\alpha^{\s(g)}(h)\dd\alpha^x(g)
  = (\alpha\circ \lambda_2) f(x),\\
  \label{eq:vertex_maps_integrals1}
  (\alpha\circ \lambda_0) f(x)
  &= \int_{G^x} \int_{G_x} f(g,h)
  \dd\tilde{\alpha}_x(g)\dd\alpha^x(h)\\
  \notag &= \int_{G_x} \int_{G^x} f(g,h)
  \dd\alpha^x(h)\dd\tilde{\alpha}_x(g)
  = (\tilde{\alpha}\circ \lambda_2) f(x),\\
  \label{eq:vertex_maps_integrals2}
  (\tilde{\alpha}\circ \lambda_0) f(x)
  &= \int_{G_x} \int_{G_{\rg(h)}} f(g,h)
  \dd\tilde{\alpha}_{\rg(h)}(g)\dd\tilde{\alpha}_x(h)\\
  \notag &= \int_{G_x} \int_{G_x} f(k h^{-1}, h)
  \dd\tilde{\alpha}_x(h)\dd\tilde{\alpha}_x(k)
  = (\tilde{\alpha}\circ \lambda_1) f(x).
\end{align}
We define continuous families of measures~\(\mu_i\) along
\(v_i\colon G^2\onto G^0\) for \(i=0,1,2\) by
\begin{align}
  \label{eq:compose_integration_maps0}
  \mu_0 \defeq \alpha\circ \lambda_1
  &= \alpha\circ \lambda_2,\\
  \label{eq:compose_integration_maps1}
  \mu_1 \defeq \alpha\circ \lambda_0
  &= \tilde{\alpha}\circ \lambda_2,\\
  \label{eq:compose_integration_maps2}
  \mu_2 \defeq \tilde{\alpha}\circ \lambda_0
  &= \tilde{\alpha}\circ \lambda_1.
\end{align}
We remember the following consequence of
\eqref{eq:lambda_definition1} and~\eqref{eq:vertex_maps_integrals2}
for later:
\begin{equation}
  \label{eq:compare_integrals_rep}
  \iint f(g,g^{-1} k) \dd\alpha^{\rg(k)}(g)\dd\tilde\alpha_x(k)
  =
  \iint f(g,h) \dd\tilde\alpha_{\rg(h)}(g)\dd\tilde\alpha_x(h)
\end{equation}
for all \(x\in G^0\) and \(f\in \Contc(G^1\times_{\s,G^0,\rg}
G_x)\).

As in Definition~\ref{def:corr_from_measures}, we assign
\cstar{}correspondences to all the families of measures above.
This gives the diagram of \cstar{}correspondences in
Figure~\ref{fig:corr_G2_G0}.  It commutes up to canonical
isomorphisms of \cstar{}correspondences.

\begin{figure}[htbp]
  \begin{tikzpicture}
    \matrix (m) [cd,column sep=4em, row sep=2.5em] {
      \Cont_0(G^1)&\Cont_0(G^0)\\
      \Cont_0(G^0)&\Cont_0(G^2)&\Cont_0(G^1)\\
      &\Cont_0(G^1)&\Cont_0(G^0)\\
    };
    \meascor[0]{m-2-2}{d_0}{\lambda_0}{m-1-1};
    \meascor{m-2-2}{d_1}{\lambda_1}{m-2-3};
    \meascor{m-2-2}{d_2}{\lambda_2}{m-3-2};
    \meascor[0]{m-2-2}{v_0}{\mu_0}{m-3-3};
    \meascor{m-2-2}{v_1}{\mu_1}{m-2-1};
    \meascor{m-2-2}{v_2}{\mu_2}{m-1-2};

    \meascor{m-2-3}{\rg}{\alpha}{m-3-3};
    \meascor{m-3-2}{\rg}{\alpha}{m-3-3};

    \meascor[0]{m-3-2}{\s}{\tilde{\alpha}}{m-2-1};
    \meascor{m-1-1}{\rg}{\alpha}{m-2-1};

    \meascor{m-1-1}{\s}{\tilde{\alpha}}{m-1-2};
    \meascor[0]{m-2-3}{\s}{\tilde{\alpha}}{m-1-2};
  \end{tikzpicture}
  \caption{Canonical isomorphisms of the \cstar{}correspondences
    associated to the continuous families of measures along the maps
    \(G^2\to G^1\to G^0\).  For each triangle,
    \eqref{eq:compose_integration_maps0}--\eqref{eq:compose_integration_maps2}
    and Lemma~\ref{lem:compose_measure_family} give a canonical
    isomorphism between the two \cstar{}correspondences that form
    the boundary of the triangle.}
  \label{fig:corr_G2_G0}
\end{figure}

Let~\(D\) be a \cstar{}algebra and~\(\Hilm[F]\) a Hilbert
\(D\)\nb-module.  A \emph{representation} of the groupoid
\cstar{}algebra~\(\Cst(G,\alpha)\) on~\(\Hilm[F]\) is a
nondegenerate \Star{}homomorphism
\(\Cst(G,\alpha)\to\Bound(\Hilm[F])\).  Our main theorem,
Theorem~\ref{the:universal_groupoid_Haus}, says that these
representations are equivalent in a precise sense to
``representations'' of~\((G,\alpha)\) on~\(\Hilm[F]\).
Representations of~\((G,\alpha)\) have some data, which is subject
to a condition.  We shall write down the data and conditions
succinctly as diagrams in the correspondence bicategory.  First,
however, we formulate them in a more pedestrian way.

\begin{definition}
  \label{def:representation1}
  The \emph{data of a representation} of \((G,\alpha)\)
  on~\(\Hilm[F]\) is a pair~\((\varphi, U)\), where \(\varphi\colon
  \Cont_0(G^0)\to\Bound(\Hilm[F])\) is a representation -- which
  turns~\(\Hilm[F]\) into a \(\Cont_0(G^0),D\)-correspondence --
  and~\(U\) is a unitary operator
  \begin{equation}
    \label{eq:Unitary-rep}
    U\colon \Lt^2(G^1,\s,\tilde{\alpha}) \otimes_\varphi \Hilm[F] \congto
    \Lt^2(G^1,\rg,\alpha) \otimes_\varphi \Hilm[F]
  \end{equation}
  that intertwines the left actions of~\(\Cont_0(G^1)\) on these
  Hilbert \(D\)\nb-modules.
\end{definition}

More briefly, \(U\) is an isomorphism of
\(\Cont_0(G^1),D\)-correspondences or, equivalently, a
\(2\)\nb-arrow in the correspondence bicategory.  It makes the
following diagram in the correspondence bicategory commute:
\begin{equation}
  \label{eq:U_diagram}
  \begin{tikzpicture}[baseline=(current bounding box.west)]
    \matrix (m) [cd,column sep=4em, row sep=2.5em] {
      \Cont_0(G^1)&\Cont_0(G^0)\\
      \Cont_0(G^0)&D\\
    };
    \meascor{m-1-1}{\s}{\tilde{\alpha}}{m-1-2}
    \meascor{m-1-1}{\rg}{\alpha}{m-2-1}
    \meascor{m-2-1}{\Hilm[F]}{}{m-2-2}
    \meascor{m-1-2}{\Hilm[F]}{}{m-2-2}
    \draw[dar,mid] (m-1-2) -- node[narrowfill] {\(U\)} (m-2-1);
  \end{tikzpicture}
\end{equation}

To define a representation of~\((G,\alpha)\), we build several other
correspondence isomorphisms out of~\(U\).  First, \(U\) induces an
isomorphism of \(\Cont_0(G^2),D\)-correspondences
\begin{multline*}
  1\otimes U\colon
  \Lt^2(G^2,d_1,\lambda_1) \otimes_{\Cont_0(G^1)}
  (\Lt^2(G^1,\s,\tilde{\alpha}) \otimes_\varphi \Hilm[F])
  \\\congto
  \Lt^2(G^2,d_1,\lambda_1) \otimes_{\Cont_0(G^1)}
  (\Lt^2(G^1,\rg,\alpha) \otimes_\varphi \Hilm[F]).
\end{multline*}
The associativity of the composition of correspondences and the
canonical isomorphisms \(\Lt^2(G^2,d_1,\lambda_1)
\otimes_{\Cont_0(G^1)} \Lt^2(G^1,\s,\tilde{\alpha}) \cong
\Lt^2(G^2,v_2,\mu_2)\) and \(\Lt^2(G^2,d_1,\lambda_1)
\otimes_{\Cont_0(G^1)} \Lt^2(G^1,\rg,\alpha) \cong
\Lt^2(G^2,v_0,\mu_0)\) in Figure~\ref{fig:corr_G2_G0} turn this into
an isomorphism
\begin{equation}
  \label{eq:d1_U}
  d_1^*(U)\colon \Lt^2(G^2,v_2,\mu_2) \otimes_\varphi \Hilm[F]
  \xrightarrow{\cong}
  \Lt^2(G^2,v_0,\mu_0) \otimes_\varphi \Hilm[F].
\end{equation}
Similarly, we define isomorphisms of
\(\Cont_0(G^2),D\)-correspondences
\begin{align}
  \label{eq:d0_U}
  d_0^*(U)\colon
  \Lt^2(G^2,v_2,\mu_2) \otimes_\varphi \Hilm[F]
  \xrightarrow{\cong}
  \Lt^2(G^2,v_1,\mu_1) \otimes_\varphi \Hilm[F],\\
  \label{eq:d2_U}
  d_2^*(U)\colon
  \Lt^2(G^2,v_1,\mu_1) \otimes_\varphi \Hilm[F]
  \xrightarrow{\cong}
  \Lt^2(G^2,v_0,\mu_0) \otimes_\varphi \Hilm[F].
\end{align}

\begin{definition}
  \label{def:representation}
  A pair \((\varphi,U)\) as in Definition~\ref{def:representation1}
  is a \emph{representation} if it also satisfies the condition
  \begin{equation}
    \label{eq:di_U-condition}
    d_1^*(U) = d_2^*(U)\circ d_0^*(U).
  \end{equation}
\end{definition}

The construction of~\(d_1^*(U)\) only involves horizontal and
vertical products of \(2\)\nb-arrows in the correspondence
bicategory.  This is described succinctly in the right diagram in
Figure~\ref{fig:U_coherence}.
\begin{figure}[htbp]
  \begin{tikzpicture}[scale=2]
    \node (G2) at (0,1.5) {\(\Cont_0(G^2)\)};
    \node (G0a) at (1,0) {\(\Cont_0(G^0)\)};
    \node (G1b) at (1,1) {\(\Cont_0(G^1)\)};
    \node (G1c) at (1,2) {\(\Cont_0(G^1)\)};
    \node (G0d) at (1,3) {\(\Cont_0(G^0)\)};
    \node (G0e) at (2,1.5) {\(\Cont_0(G^0)\)};
    \node (D) at (3,1.5) {\(D\)};
    \meascor[0]{G2}{v_0}{\mu_0}{G0a}
    \meascor[0]{G2}{d_2}{\lambda_2}{G1b}
    \meascor[0]{G2}{d_0}{\lambda_0}{G1c}
    \meascor[0]{G2}{v_2}{\mu_2}{G0d}
    \meascor[0]{G2}{v_1}{\mu_1}{G0e}
    \meascor{G1b}{r}{\alpha}{G0a}
    \meascor[0]{G1b}{s}{\tilde{\alpha}}{G0e}
    \meascor{G1c}{s}{\tilde{\alpha}}{G0d}
    \meascor[0]{G1c}{r}{\alpha}{G0e}
    \meascor[0]{G0a}{\Hilm[F]}{}{D}
    \meascor[0]{G0d}{\Hilm[F]}{}{D}
    \meascor{G0e}{\Hilm[F]}{}{D}
    \draw[dar,mid] (G0e) -- node[narrowfill] {\(U\)} (G0a);
    \draw[dar,mid] (G0d) -- node[narrowfill] {\(U\)} (G0e);
  \end{tikzpicture}
  \begin{tikzpicture}[scale=2]
    \node (G2) at (0,1) {\(\Cont_0(G^2)\)};
    \node (G0a) at (1,0) {\(\Cont_0(G^0)\)};
    \node (G1b) at (1,1) {\(\Cont_0(G^1)\)};
    \node (G0c) at (1,2) {\(\Cont_0(G^0)\)};
    \node (D) at (2,1) {\(D\)};
    \meascor[0]{G2}{v_0}{\mu_0}{G0a}
    \meascor{G2}{d_1}{\lambda_1}{G1b}
    \meascor[0]{G2}{v_2}{\mu_2}{G0c}
    \meascor{G1b}{r}{\alpha}{G0a}
    \meascor{G1b}{s}{\tilde{\alpha}}{G0c}
    \meascor[0]{G0a}{}{\Hilm[F]}{D}
    \meascor[0]{G0c}{\Hilm[F]}{}{D}
    \draw[dar,mid] (G0c) to[bend left=40] node[narrowfill] {\(U\)} (G0a);
  \end{tikzpicture}
  \caption{Two parallel isomorphisms of correspondences constructed
    from~\(U\).  Each triangle or square corresponds to one
    isomorphism of \cstar{}correspondences.  The unlabelled
    ones involve the canonical isomorphisms of
    \cstar{}correspondences in Figure~\ref{fig:corr_G2_G0}.}
  \label{fig:U_coherence}
\end{figure}
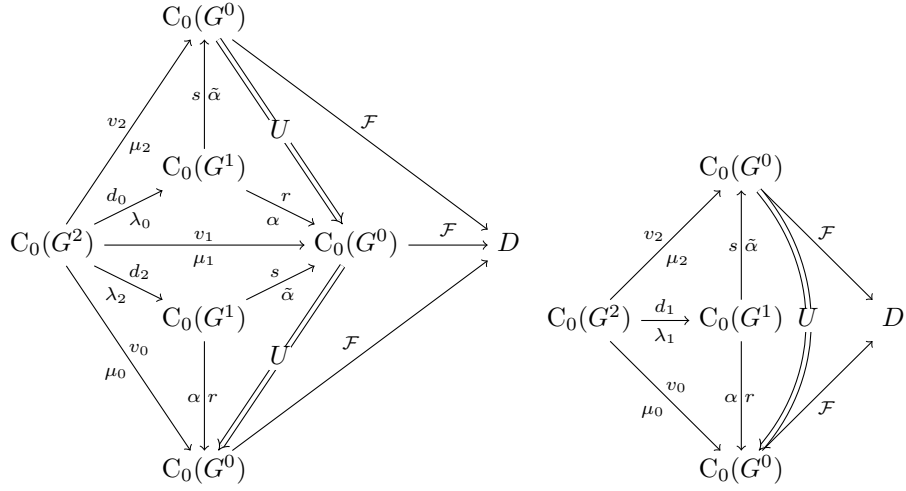
Each triangle or quadrilateral in this diagram describes a
\(2\)\nb-arrow between certain arrows.  These are pasted together
using the appropriate horizontal or vertical products to obtain a
\(2\)\nb-arrow from the composite arrow \(\Lt^2(G^2,v_2,\mu_2)
\otimes_\varphi \Hilm[F]\) in the top to the composite arrow
\(\Lt^2(G^2,v_0,\mu_0) \otimes_\varphi \Hilm[F]\) in the bottom.
Similarly, the left diagram in Figure~\ref{fig:U_coherence}
describes \(d_2^*(U)\circ d_0^*(U)\).  So~\eqref{eq:di_U-condition}
says that the \(2\)\nb-arrows \(\Lt^2(G^2,v_2,\mu_2) \otimes_\varphi
\Hilm[F] \Rightarrow \Lt^2(G^2,v_0,\mu_0) \otimes_\varphi \Hilm[F]\)
described by the two diagrams in Figure~\ref{fig:U_coherence} are
equal.

\subsection{The regular representation}
\label{sec:regular-rep}

We may combine the left regular representations
of~\(\Cst(G,\alpha)\) on the Hilbert spaces
\(\Lt^2(G_x,\tilde\alpha_x)\) for \(x\in G^0\) into a single
representation on the Hilbert \(\Cont_0(G^0)\)-module
\(\Hilm[F]\defeq \Lt^2(G^1,\s,\tilde\alpha)\).  To illustrate our
definition, we describe the corresponding representation
of~\((G,\alpha)\) on~\(\Hilm[F]\).  We equip~\(\Hilm[F]\) with the
left action~\(\varphi\) of~\(\Cont_0(G^0)\) defined so
that~\(\varphi(f)\) for \(f\in\Cont_0(G^0)\) acts by pointwise
multiplication with the function \(f\circ\rg\in\Contb(G^1)\).
By abuse of notation, we still write~\(\Hilm[F]\) for the
\(\Cont_0(G^0),\Cont_0(G^0)\)-correspondence
\((\Hilm[F],\varphi)\), which is the \cstar{}correspondence
\(\rg^* \Lt^2(G^1,\s,\tilde\alpha)\) induced by the topological
correspondence
\[
G^0 \xleftarrow{\rg} G^1 \xrightarrow[\tilde\alpha]{\s} G^0.
\]

By Proposition~\ref{prop:functoriality-of-corr}, the
tensor product
\(\Lt^2(G^1,\s,\tilde{\alpha})\otimes_\varphi\Hilm[F]\) is
isomorphic to the \cstar{}correspondence from \(\Cont_0(G^1)\) to
\(\Cont_0(G^0)\) associated to the composite of the underlying
topological correspondences.  This involves the fibre product space
\(G^1 \times_{\s,G^0,\rg} G^1 = G^2\) and the maps \(d_2\colon G^2
\onto G^1\), \((g,h)\mapsto g\), and \(v_2\colon G^2 \onto G^0\),
\((g,h)\mapsto \s(h)\).  The measure family along~\(v_2\) for this
composite is, by definition, \(\tilde\alpha\circ\lambda_0 =
\mu_2\). Thus
\(\Lt^2(G^1,\s,\tilde{\alpha})\otimes_\varphi\Hilm[F] \cong d_2^*
\Lt^2(G^2,v_2,\mu_2)\) is induced by the topological
correspondence
\begin{equation}
  \label{eq:left_regular_source}
  G^1 \xleftarrow{d_2} G^2 \xrightarrow[\mu_2]{v_2} G^0.
\end{equation}
Similarly, \(\Lt^2(G^1,\rg,\alpha)\otimes_\varphi\Hilm[F]\)
is isomorphic to the \cstar{}correspondence induced by
the topological correspondence
\begin{equation}
  \label{eq:left_regular_range}
  G^1 \xleftarrow{\pr_1} G^1\times_{\rg,G^0,\rg} G^1
  \xrightarrow[\tilde\alpha\circ\alpha]{\s\circ\pr_2} G^0
\end{equation}
with \(\tilde\alpha\circ\alpha (f)(x) \defeq \iint
f(g,k)\dd\alpha^{\rg(k)}(g) \dd\tilde\alpha_x(k)\) for \(f\in
\Contc(G^1\times_{\rg,G^0,\rg} G^1)\), \(x\in G^0\).  We claim that
the topological correspondences \eqref{eq:left_regular_source}
and~\eqref{eq:left_regular_range} are isomorphic through the
homeomorphism
\begin{equation}
  \label{eq:fundamental-homeo}
  \Upsilon\colon G^2 = G^1\times_{\s,G^0,\rg} G^1 \congto
  G^1\times_{\rg,G^0,\rg} G^1,\qquad
  (g,h)\mapsto (g,g\cdot h).
\end{equation}
The conditions \(\pr_1\circ \Upsilon = d_2\) and \(\s\circ
\pr_2\circ \Upsilon = v_2\) are trivial.
Equation~\eqref{eq:compare_integrals_rep} says that
\(\mu_2=\Upsilon^{-1}_*(\tilde\alpha\circ\alpha)\) or, equivalently,
\(\Upsilon_*(\mu_2)=\tilde\alpha\circ\alpha\).  Thus~\(\Upsilon\) is
an isomorphism of topological correspondences.  It induces an
isomorphism of \cstar{}\alb{}correspondences
\[
U\colon
\Lt^2(G^1,\s,\tilde\alpha)\otimes_\varphi\Hilm[F]\congto
\Lt^2(G^1,\rg,\alpha)\otimes_\varphi\Hilm[F].
\]
We claim that~\((\varphi,U)\) is a representation of~\((G,\alpha)\).

The \cstar{}correspondence \(\Lt^2(G^2,v_i,\mu_i)\otimes_\varphi
\Hilm[F]\) from~\(\Cont_0(G^2)\) to~\(\Cont_0(G^0)\) is associated
to the topological correspondence
\[
G^2 \xleftarrow{\pr_1} G^2\times_{v_i,G^0,\rg} G^1
\xrightarrow[\tilde\alpha\circ\mu_i]{\s\circ\pr_2} G^0.
\]
The isomorphisms~\(d_i^*(U)\) in
\eqref{eq:d1_U}--\eqref{eq:d2_U} are induced by isomorphisms of
topological correspondences, namely, the homeomorphisms
\begin{alignat*}{2}
  d_1^*(\Upsilon)&\colon G^2\times_{v_2,G^0,\rg} G^1\to
  G^2\times_{v_0,G^0,\rg} G^1,&\qquad
  (g,h,l)&\mapsto (g,h,ghl),\\
  d_0^*(\Upsilon)&\colon G^2\times_{v_2,G^0,\rg} G^1\to
  G^2\times_{v_1,G^0,\rg} G^1,&\qquad
  (g,h,l)&\mapsto (g,h,hl),\\
  d_2^*(\Upsilon)&\colon G^2\times_{v_1,G^0,\rg} G^1\to
  G^2\times_{v_0,G^0,\rg} G^1,&\qquad
  (g,h,l)&\mapsto (g,h,gl).
\end{alignat*}
Since the multiplication in~\(G\)
is associative, these isomorphisms of topological correspondences
satisfy \(d_2^*(\Upsilon)\circ d_0^*(\Upsilon) = d_1^*(\Upsilon)\).
Therefore, \(d_2^*(U)\circ d_0^*(U) = d_1^*(U)\).
Hence~\((\varphi,U)\) is a representation of~\((G,\alpha)\).

\subsection{Formulation of the universal property}
\label{sec:formulate_universal}

The groupoid \cstar{}algebra~\(\Cst(G,\alpha)\) is defined
in~\cite{Renault:Groupoid_Cstar}*{Chapter II}.  The
following theorem uses this definition, but it may also serve as
an alternative definition of~\(\Cst(G,\alpha)\).

\begin{theorem}
  \label{the:universal_groupoid_Haus}
  Let~\(G\) be a locally compact, Hausdorff groupoid with a Haar
  system~\(\alpha\).  Let~\(D\) be a \cstar{}algebra
  and~\(\Hilm[F]\) a Hilbert \(D\)\nb-module.  There is a bijection
  between representations of~\(\Cst(G,\alpha)\) on~\(\Hilm[F]\) and
  representations of~\((G,\alpha)\) on~\(\Hilm[F]\).  It is natural
  in the following two ways:
  \begin{enumerate}
  \item \label{en:universal_groupoid_Haus1} Let \(\Hilm[F]_1\)
    and~\(\Hilm[F]_2\)
    be two Hilbert \(D\)\nb-modules
    and let \(V\colon \Hilm[F]_1\into\Hilm[F]_2\)
    be an isometry.  Then it
    intertwines two representations of~\(\Cst(G,\alpha)\)
    on \(\Hilm[F]_1\)
    and~\(\Hilm[F]_2\)
    if and only if~\(V\)
    intertwines the corresponding representations of~\((G,\alpha)\).
  \item \label{en:universal_groupoid_Haus2} Let~\(\Hilm[E]\)
    be a \cstar{}correspondence
    from~\(D\)
    to a \cstar{}algebra~\(D'\).
    A representation of~\(\Cst(G,\alpha)\)
    or of~\((G,\alpha)\)
    on~\(\Hilm[F]\)
    induces a representation of~\(\Cst(G,\alpha)\)
    or of~\((G,\alpha)\)
    on \(\Hilm[F]\otimes_D\Hilm[E]\),
    respectively.  The bijections between representations of
    \(\Cst(G,\alpha)\)
    and~\((G,\alpha)\)
    on \(\Hilm[F]\)
    and~\(\Hilm[F]\otimes_D\Hilm[E]\)
    are compatible with these induction processes.
  \end{enumerate}
  This universal property characterises~\(\Cst(G,\alpha)\)
  uniquely up to canonical isomorphism.
\end{theorem}

The two naturality properties in the theorem above are the
same as in the definition of a \(\Cst\)\nb-hull for a class of
integrable representations of a \Star{}algebra in
\cite{Meyer:Unbounded}*{Section~3}.

Before we prove Theorem~\ref{the:universal_groupoid_Haus}, we relate
it to Renault's Integration and Disintegration Theorems
in~\cite{Renault:Representations}. Thus we specialise to the case
where \(D=\C\), \(\Hilm[F]\) is a separable Hilbert space,
and~\(G^1\) is second countable.
Let \((\varphi, U)\) be a representation of \((G,\alpha)\)
on~\(\Hilm[F]\) as in Theorem~\ref{the:universal_groupoid_Haus}.
The representation theory of commutative \cstar{}algebras on
separable Hilbert spaces is well known (see, for instance,
\cite{Dixmier:Cstar-algebres}*{Sections 8.2--3} or
  \cite{Kadison-Ringrose:Fundamentals_2}*{Chapter~14}).  When
we apply it to the representation~\(\varphi\) of~\(\Cont_0(G^0)\),
we get a measure class~\([\nu]\) on~\(G^0\), a
\([\nu]\)\nb-measurable field of Hilbert spaces \(\Hils=
(\Hils_x)_{x\in G^0}\) on~\(G^0\), and a unitary operator \(\Hilm[F]
\congto \Lt^2(G^0,\nu,\Hils)\) that intertwines~\(\varphi\) and the
representation of~\(\Cont_0(G^0)\) on \(\Lt^2(G^0,\nu,\Hils)\) by
pointwise multiplication; here \(\Lt^2(G^0,\nu,\Hils)\) is
the Hilbert space of (equivalence classes of) square-integrable
sections of~\(\Hils\) with respect to~\(\nu\), which is also
called the direct integral of the field \(\Hils=(\Hils_x)_{x\in
  G^0}\) and denoted by \(\int_{G^0}^\oplus \Hils_x \dd\nu(x)\),
see~\cite{Dixmier:Cstar-algebres}.

Tensoring \(\Lt^2(G^0,\nu,\Hils)\) with the \cstar{}correspondences
\(\Lt^2(G^1,\s,\tilde{\alpha})\) and \(\Lt^2(G^1,\rg,\alpha)\) gives
representations of~\(\Cont_0(G^1)\) on other separable
Hilbert spaces, which may be described similarly.  We may compute
the measure classes and measurable fields on~\(G^1\) directly.  For
\(\Lt^2(G^1,\s,\tilde{\alpha})\), we get the measure class
\([\nu\circ\tilde{\alpha}]\) and the pullback field~\(\s^*(\Hils)\)
with fibre~\(\Hils_{\s(g)}\) at \(g\in G^1\).  For
\(\Lt^2(G^1,\rg,\alpha)\), we get the measure class
\([\nu\circ\alpha]\) and the pullback field~\(\rg^*(\Hils)\) with
fibre~\(\Hils_{\rg(g)}\) at \(g\in G^1\): the proof is
  similar to that of Proposition~\ref{prop:functoriality-of-corr}.
A unitary intertwiner~\(U\) between these two representations
of~\(\Cont_0(G^1)\) can only exist if the measure classes
\([\nu\circ\tilde{\alpha}]\) and~\([\nu\circ\alpha]\) on~\(G^1\) are
equal (see
  \cite{Dixmier:Cstar-algebres}*{Proposition~8.2.4}).  By
definition, this means that the measure~\(\nu\) on~\(G^0\) is
\emph{quasi-invariant}.  The description of representations
of~\(\Cont_0(G^1)\) through a measure class and a measurable field
of Hilbert spaces is natural in the formal sense.  That is, any
unitary intertwiner between the two representations
of~\(\Cont_0(G^1)\) on \(\Lt^2(G^1,\nu\circ\alpha,\s^*\Hils)\) and
\(\Lt^2(G^1,\nu\circ\tilde{\alpha},\rg^*\Hils)\) must have the form
\begin{equation}
  \label{eq:intertwiner_G1}
  (U f)(g) = U_g (f(g)) \cdot
  \sqrt{\frac{\diff \nu\circ\alpha}
    {\diff \nu\circ\tilde{\alpha}}}
\end{equation}
for \(f\in\Lt^2(G^1,\nu\circ\alpha,\s^*\Hils)\) and almost all
\(g\in G^1\), where \((U_g)_{g\in G^1}\) is an isomorphism between
the measurable fields of Hilbert spaces \(\s^*\Hils\)
and~\(\rg^*\Hils\).  That is, \(U_g\) is a unitary operator
\(U_g\colon \Hils_{\s(g)} \congto \Hils_{\rg(g)}\) for all \(g\in G^1\)
outside a null set for the measure class
\([\nu\circ\alpha]\).  The Radon--Nikodym derivative
in~\eqref{eq:intertwiner_G1} is well defined because
\([\nu\circ\tilde{\alpha}]=[\nu\circ\alpha]\).  Conversely, any
measurable family of unitary operators
\(U_g\colon \Hils_{\s(g)} \congto \Hils_{\rg(g)}\) defines a
unitary intertwiner between the two representations
of~\(\Cont_0(G^1)\) on \(\Lt^2(G^1,\nu\circ\alpha,\s^*\Hils)\) and
\(\Lt^2(G^1,\nu\circ\tilde{\alpha},\rg^*\Hils)\).

Similarly, we may identify
\[
\Lt^2(G^2,\mu_j,v_j) \otimes_{\Cont_0(G^0)} \Lt^2(G^0,\nu,\Hils)
\cong \Lt^2(G^2,\nu\circ\mu_j,v_j^*\Hils);
\]
that is, we take \(\Lt^2\)\nb-sections of the measurable fields of
Hilbert spaces with fibres \(\Hils_{\rg(g)}\), \(\Hils_{\s(g)}=
\Hils_{\rg(h)}\) and \(\Hils_{\s(h)}\) at \((g,h)\in G^2\) with
respect to the measures \(\nu\circ\mu_j\) for \(j=0,1,2\),
respectively.  The isomorphisms of
\cstar{}correspondences~\(d_i^*(U)\) for \(i=0,1,2\) become
isomorphisms
\begin{alignat*}{2}
  d_0^*(U)&\colon \Lt^2(G^2,\mu_2,v_2) \congto \Lt^2(G^2,\mu_1,v_1),
  &\qquad d_0^*(U) &\sim (U_h)_{(g,h)\in G^2},\\
  d_1^*(U)&\colon \Lt^2(G^2,\mu_2,v_2) \congto \Lt^2(G^2,\mu_0,v_0),
  &\qquad d_1^*(U) &\sim (U_{g h})_{(g,h)\in G^2},\\
  d_2^*(U)&\colon \Lt^2(G^2,\mu_1,v_1) \congto \Lt^2(G^2,\mu_0,v_0),
  &\qquad d_2^*(U) &\sim (U_g)_{(g,h)\in G^2},
\end{alignat*}
where~\(\sim\) means that the unitary \(d_0^*(U)\) corresponds to
the measurable family of unitary operators \(U_h\colon \Hils_{\s(h)}
\congto \Hils_{\rg(h)}\), and so on, as
in~\eqref{eq:intertwiner_G1}.  Equation~\eqref{eq:di_U-condition}
holds if and only if \(U_{gh} = U_g\circ U_h\) for almost all
\((g,h)\in G^2\) with respect to the measure class \([\nu\circ
\mu_0] = [\nu\circ \mu_1] = [\nu\circ \mu_2]\).  The Radon--Nikodym
derivatives in~\eqref{eq:intertwiner_G1} cancel automatically.

Summing up, the representation theory of commutative \cstar{}algebras
translates a representation~\((\varphi,U)\)
of~\((G,\alpha)\)
on a separable Hilbert space in Definition~\ref{def:representation}
into
\begin{enumerate}
\item a quasi-invariant measure class~\([\nu]\) on~\(G^0\),
  that is, \([\nu\circ \alpha]=[\nu\circ\tilde\alpha]\);
\item a \([\nu]\)\nb-measurable field of Hilbert spaces~\(\Hils\)
  on~\(G^0\);
\item unitary operators \(U_g\colon \Hils_{\s(g)} \congto
  \Hils_{\rg(g)}\) for \([\nu\circ\alpha]\)-almost all \(g\in
  G^1\), which satisfy \(U_{gh} = U_g\circ U_h\) for
  \([\nu\circ\mu_i]\)-almost all \((g,h)\in G^2\).
\end{enumerate}

Comparing this with Renault's notion of a representation in
\cite{Renault:Representations}*{Definition 3.4}, there is only one
technical difference about which null sets are allowed in \(G^1\)
and~\(G^2\).  In~\cite{Renault:Representations}, there is one
\([\nu]\)\nb-negligible subset~\(N\) of~\(G^0\) such that the
unitaries~\(U_g\) are defined and satisfy \(U_{gh} = U_g\circ U_h\)
whenever \((g,h)\in G^2\) and \(\rg(g),\s(g)=\rg(h),\s(h)\notin N\).
This extra information is often useful, but it is not needed to
integrate a representation.  The formulas for integration of
representations in~\cite{Renault:Representations} still work if we
allow arbitrary null sets in~\(G^1\) and~\(G^2\).  If~\(G^1\) is
second countable, then \cite{Ramsay:Topologies}*{Lemma~3.3} allows to
modify~\((U_g)\) on a set of measure~\(0\) so that there is a null
set~\(N\) as above.

\begin{remark}
  \label{rem:no_disintegration_module}
  The theory in~\cite{Renault:Representations} is limited to
  Hilbert space representations because there is no disintegration
  theory for representations of commutative \cstar{}algebras on
  Hilbert modules.  For instance, there is no reasonable way to
  define ``fibres'' for the identity representation of
  \(\Cont[0,1]\) on itself.  The fibre at \(t\in [0,1]\) ought to
  have dimension~\(0\) away from~\(t\).  There is, however, no
  continuous field of Hilbert spaces over~\([0,1]\) with exactly
  one non-zero fibre.
\end{remark}

What are the Hilbert space representations of~\((G,\alpha)\)
if~\(G\)
is a group?  Since~\(G^0\)
has only one point, the quasi-invariant measure on~\(G^0\)
is irrelevant and the measurable field over~\(G^0\)
is simply the Hilbert space~\(\Hils\)
on which the representation takes place.  The isomorphism of
correspondences~\(U\)
is a unitary intertwiner for the pointwise multiplication action
of~\(\Cont_0(G)\)
on \(L^2(G,\Hils) \cong L^2(G)\otimes \Hils\).
This is equivalent to a measurable family of unitary operators
\(U_g\in\U(\Hils)\).
The condition \(d_2^*(U)\circ d_0^*(U) = d_1^*(U)\)
holds if and only if \(U_g \circ U_h = U_{g\cdot h}\)
for almost all \((g,h)\in G^2\).
Thus~\((U_g)_{g\in G}\)
is a measurable weak representation of~\(G\)
on~\(\Hils\)
(compare \cite{Ramsay:Topologies} for the notation of weak
representations).  The usual universal property of~\(\Cst(G)\)
uses \emph{continuous} representations.  Hence both universal properties
together imply the following:

\begin{corollary}
  \label{cor:automatic-continuity}
  Any measurable weak representation of a locally compact group is equal almost
  everywhere to a continuous group representation.
\end{corollary}

\begin{proof}
  First, any measurable weak representation of~\(G\)
  integrates to a nondegenerate representation of the convolution
  algebra~\(L^1(G)\).
  Secondly, any nondegenerate Banach \(L^1(G)\)-module
  comes from a continuous representation of~\(G\)
  because the regular representation on~\(L^1(G)\)
  is continuous.  Third, this continuous representation must be equal
  almost everywhere to the given weak representation in order to
  integrate to the same representation of~\(L^1(G)\).
\end{proof}

\begin{remark}
  Our universal property also works for non-separable Hilbert spaces.
  But unitary intertwiners on~\(L^2(G,\Hils)\)
  are no longer equivalent to measurable families of unitary operators
  on~\(\Hils\)
  up to equality almost everywhere.  For instance, consider the family
  of unitary operators~\(U_t\)
  on \(\ell^2(\R)\),
  where \(U_t(\delta_t) \defeq \delta_{t+1}\),
  \(U_t(\delta_{t+1}) \defeq \delta_t\),
  and \(U_t(\delta_s)\defeq \delta_s\)
  for \(s\in\R\setminus\{t,t+1\}\).
  This family describes the identity operator on \(L^2(\R,\ell^2\R)\)
  because if \(s\in\R\),
  then \(U_t(\delta_s) = \delta_s\)
  for almost all \(t\in\R\).
  Nevertheless, the set of \(t\in\R\)
  with \(U_t=1\) is empty.
\end{remark}

What happens if~\(G\) is a locally compact space viewed as a
groupoid?  In this case, \(\s=\rg\) and \(\alpha=\tilde{\alpha}\).
Equation~\eqref{eq:di_U-condition} says that \(U\cdot U = U\).
Since~\(U\) is unitary, we may cancel~\(U\) here, so~\(U\)
is the identity map.  Thus a representation in the sense of
Definition~\ref{def:representation} is simply a representation
of~\(\Cont_0(G^0)\), which is also the groupoid \cstar{}algebra.
So Theorem~\ref{the:universal_groupoid_Haus} is trivial in this
case.  In contrast, the Integration and Disintegration Theorems
in~\cite{Renault:Representations} are non-trivial even for spaces
viewed as groupoids.

The proof of Theorem~\ref{the:universal_groupoid_Haus} requires two
constructions.  Integration takes a representation of~\((G,\alpha)\)
to one of~\(\Cst(G,\alpha)\), and disintegration takes a
representation of~\(\Cst(G,\alpha)\) to one of~\((G,\alpha)\).  We
shall discuss these two constructions in Sections \ref{sec:integrate}
and~\ref{sec:disintegrate}, respectively.  We prove that they are
inverse to each
other in Section~\ref{sec:back_and_forth}.

\section{Integration}
\label{sec:integrate}

Let~\((\varphi,U)\)
be a representation of~\((G,\alpha)\)
on a Hilbert module~\(\Hilm[F]\)
over a \cstar{}algebra~\(D\).
We are going to ``integrate'' it to a representation
of~\(\Cst(G,\alpha)\).

\begin{definition}[(Creation operators)]
  \label{def:creation}
  Let~\(\Hilm\)
  be a Hilbert module over a \cstar{}algebra~\(B\),
  let~\(\Hilm[F]\)
  be a \(B\)-\(D\)-correspondence,
  and \(x\in\Hilm\).
  Let \(T_x\colon \Hilm[F]\to \Hilm\otimes_B \Hilm[F]\)
  denote the operator \(y\mapsto x\otimes y\).
  Its adjoint~\(T_x^*\)
  maps \(z\otimes y \mapsto \braket{x}{z}\cdot y\)
  for \(x,z\in\Hilm\), \(y\in\Hilm[F]\).
\end{definition}

Fix \(f\in \Contc(G^1)\), and let~\(M_f\) denote the
operator of pointwise multiplication by~\(f\), which is how
\(\Cont_0(G^1)\) acts in the \cstar{}correspondences
\(\Lt^2(G^1,\s,\tilde{\alpha})\) and \(\Lt^2(G^1,\rg,\alpha)\).
Choose functions \(h_1,h_2\in \Contc(G^1)\) with
\(h_1(g)=h_2(g)=1\) for all \(g\in\supp f\).  View \(h_1\)
and~\(h_2\) as elements of \(\Lt^2(G^1,\rg,\alpha)\) and
\(\Lt^2(G^1,\s,\tilde{\alpha})\), respectively.  Let
\(L(f)\in\Bound(\Hilm[F])\) be the composite operator
\[
\begin{tikzpicture}
  \matrix (m) [cd] {
    \Hilm[F]&
    \Lt^2(G^1,\s,\tilde{\alpha})\otimes_\varphi \Hilm[F]&&
    \Lt^2(G^1,\s,\tilde{\alpha})\otimes_\varphi \Hilm[F]\\
    &\Lt^2(G^1,\rg,\alpha)\otimes_\varphi \Hilm[F]&&
    \Lt^2(G^1,\rg,\alpha)\otimes_\varphi \Hilm[F]&
    \Hilm[F].\\
  };
  \draw[cdar] (m-1-1) -- node {\(T_{h_2}\)} (m-1-2);
  \draw[cdar] (m-1-2) -- node {\(M_{f}\otimes\Id_{\Hilm[F]}\)} (m-1-4);
  \draw[cdar] (m-2-2) -- node {\(M_{f}\otimes\Id_{\Hilm[F]}\)} (m-2-4);
  \draw[cdar] (m-1-2) -- node {\(U\)} node[swap] {\(\cong\)} (m-2-2);
  \draw[cdar] (m-1-4) -- node {\(U\)} node[swap] {\(\cong\)} (m-2-4);
  \draw[cdar] (m-2-4) -- node {\(T_{h_1}^*\)} (m-2-5);
\end{tikzpicture}
\]
The square commutes because~\(U\), as an isomorphism of
\cstar{}correspondences, intertwines the left actions
of~\(\Cont_0(G^1)\).  We write~\(M_f\) instead
of~\(M_f\otimes\Id_{\Hilm[F]}\) in the following to simplify
notation.  So we write
\begin{equation}
  \label{eq:define_Lf}
  L(f)\defeq T_{h_1}^*UM_fT_{h_2}=T_{h_1}^*M_fUT_{h_2}.
\end{equation}

\begin{lemma}
  \label{lem:Lf_well-defined}
  The operator~\(L(f)\) does not depend on \(h_1\) and~\(h_2\) and
  satisfies
  \[
  \norm{L(f)}\le
  \norm{\alpha(\abs{f})}_\infty^{1/2} \cdot
  \norm{\tilde{\alpha}(\abs{f})}_\infty^{1/2}
  \le \norm{f}_I
  \]
  for all \(f\in \Contc(G^1)\),
  where \(\norm{f}_I\defeq \max \{\norm{\alpha(\abs{f})}_\infty,
  \norm{\tilde{\alpha}(\abs{f})}_\infty \}\) and \(\alpha\)
  and~\(\tilde{\alpha}\) also denote the integration maps
  \[
  \alpha(\abs{f})(x) \defeq
  \int_{G^x} \abs{f(g)} \dd\alpha^x(g),\qquad
  \tilde{\alpha}(\abs{f})(x) \defeq
  \int_{G_x} \abs{f(g)} \dd\tilde{\alpha}_x(g).
  \]
\end{lemma}

Similar estimates are used in
\cite{Khoshkam-Skandalis:Crossed_inverse_semigroup}*{Section 3.6} to
prove that the regular representation of a groupoid is bounded by the
\(I\)\nb-norm.

\begin{proof}
  There are \(f_1,f_2\in\Contc(G^1)\)
  with \(\supp(f_i)\subseteq \supp(f)\)
  and \(f(g)=\conj{f_1}(g)\cdot f_2(g)\)
  for all \(g\in G^1\).
  A good choice for later estimates is to take
  \(f_2(g) \defeq \sqrt{\abs{f(g)}}\)
  and \(f_1(g) \defeq \conj{f}(g)/f_2(g)\)
  if \(f(g)\neq0\)
  and \(f_1(g)\defeq0\)
  if \(f(g)=0\).
  Now we use \(M_{f_i}\circ T_{h_i} = T_{f_i\cdot h_i} = T_{f_i}\)
  and that~\(M_{\conj{f_1}} = M_{f_1}^*\)
  commutes with~\(U\) to simplify
  \[
  \braket{\xi}{L(f)\eta}
  = \braket{h_1\otimes \xi}{U(\conj{f_1} f_2 h_2\otimes \eta)}
  = \braket{f_1 h_1\otimes \xi}{U(f_2 h_2\otimes \eta)}
  = \braket{f_1\otimes \xi}{U(f_2\otimes \eta)}
  \]
  for all \(\xi,\eta\in \Hilm[F]\).  That is,
  \begin{equation}
    \label{eq:integrated_form_explicit}
    L\bigl(\conj{f_1}\cdot f_2\bigr) = T_{f_1}^* U T_{f_2}
  \end{equation}
  for all \(f_1,f_2\in \Contc(G^1)\), where~\(\cdot\) denotes the
  pointwise product.  This does not depend on \(h_1\) and~\(h_2\)
  any more.  And \(\norm{T_f} = \norm{T_f^*} =
  \norm{f}\) and \(\norm{U}=1\) imply
  \[
  \norm{L(f)}
  \le \norm{f_1}_{\Lt^2(G^1,\rg,\alpha)} \norm{f_2}_{\Lt^2(G^1,\s,\tilde{\alpha})}.
  \]
  If we choose \(f_1\) and~\(f_2\) as above, then
  \(\abs{f_1(g)}^2 = \abs{f_2(g)}^2 = \abs{f(g)}\) for all \(g\in
  G^1\), so that
  \[
  \norm{f_1}^2_{\Lt^2(G^1,\rg,\alpha)} =
  \norm{\alpha(\abs{f})}_\infty,
  \qquad
  \norm{f_2}^2_{\Lt^2(G^1,\s,\tilde{\alpha})} =
  \norm{\tilde{\alpha}(\abs{f})}_\infty.
  \]
  This gives the desired norm estimate for~\(L(f)\).
\end{proof}

The previous lemma implies that the map~\(L\) is continuous
from~\(\Contc(G^1)\) with the inductive limit topology
to~\(\Bound(\Hilm[F])\) with the norm topology because the inductive
limit topology is stronger than the \(I\)\nb-norm topology.

\begin{lemma}
  \label{lem:Lf_nondegenerate}
  The linear span of~\(L(f)\xi\)
  for \(f\in\Contc(G^1)\),
  \(\xi\in\Hilm[F]\) is dense in~\(\Hilm[F]\).
\end{lemma}

\begin{proof}
  The linear span of~\(T_{f_2}\eta\) for \(f_2\in\Contc(G^1)\),
  \(\eta\in \Hilm[F]\) is dense in the tensor product
  \(\Lt^2(G^1,\s,\tilde{\alpha}) \otimes_\varphi \Hilm[F]\).
  Since~\(U\) is unitary, the linear span of~\(U T_{f_2}\eta\) for
  such \(f_2\) and~\(\eta\) is still dense in
  \(\Lt^2(G^1,\rg,\alpha)\otimes_\varphi \Hilm[F]\).  Then the
  linear span of \(L\bigl(\conj{f_1}\cdot f_2\bigr)(\eta) =
  T_{f_1}^* U T_{f_2}\eta\) for \(f_1,f_2\in\Contc(G^1)\), \(\eta\in
  \Hilm[F]\) is dense in~\(\Hilm[F]\) because the Hilbert
  \(\Cont_0(G^0)\)-module \(\Lt^2(G^1,\rg,\alpha)\) is full
  and~\(\varphi\) is nondegenerate.  Here we have
  used~\eqref{eq:integrated_form_explicit}.
\end{proof}

\begin{proposition}
  \label{pro:Lf_star-representation}
  The map~\(L\) is a nondegenerate \Star{}homomorphism from the
  convolution algebra \((\Contc(G^1),*)\) to \(\Bound(\Hilm[F])\).
\end{proposition}

\begin{proof}
  We are going to prove below that
  \begin{equation}
    \label{eq:Lf_star-rep}
    L(f_1)^*L(f_2) = L(f_1^* * f_2).
  \end{equation}
  for all \(f_1,f_2\in \Contc(G^1)\).  We first claim that this
  implies \(L(f)^* = L(f^*)\) and then \(L(f_1) L(f_2) = L(f_1 *
  f_2)\) for \(f,f_1,f_2\in\Contc(G^1)\); that is, \(L\) is a
  \Star{}representation.  Since \(f^{**}=f\) for all
  \(f\in\Contc(G^1)\), \eqref{eq:Lf_star-rep} is equivalent to
  \(L(f_1^*)^*L(f_2) = L(f_1 * f_2)\).  Since \(\Contc(G^1)\) is a
  \Star{}algebra, this implies
  \begin{multline*}
    L(f_1^*)^* L(f_2^*)^* L(f_3)
    = L(f_1^*)^* L(f_2*f_3)
    = L(f_1 * (f_2*f_3))
    = L((f_1 * f_2)*f_3)
    \\= L((f_1 * f_2)^*)^* L(f_3)
    = L(f_2^* * f_1^*)^* L(f_3)
    = L(f_1^*)^* L(f_2) L(f_3).
  \end{multline*}
  Hence \(\braket{L(f_1^*)\xi_1}{L(f_2^*)^* L(f_3)\xi_2}
  =\braket{L(f_1^*)\xi_1}{L(f_2) L(f_3)\xi_2}\) for all
  \(\xi_1,\xi_2\in\Hilm[F]\).  Vectors of the form \(L(f_1^*)\xi_1\)
  or~\(L(f_3)\xi_2\) span a dense subspace of~\(\Hilm[F]\) by
  Lemma~\ref{lem:Lf_nondegenerate}.  Hence \(L(f_2^*)^*=L(f_2)\) for
  all \(f_2\in\Contc(G^1)\) and~\(L\) is nondegenerate.

  It remains to prove~\eqref{eq:Lf_star-rep}.  Our analysis in the
  Hilbert space case suggests that we need~\eqref{eq:di_U-condition}
  with the operators~\(d_i^*(U)\) for \(i=0,1,2\) in
  \eqref{eq:d1_U}--\eqref{eq:d2_U}.  We will use the equivalent
  formula \(d_2^*(U)^* d_1^*(U) = d_0^*(U)\) because \(d_1^*(U) =
  d_2^*(U)\circ d_0^*(U)\) would lead to a proof that \(L(f_1 * f_2)
  = L(f_1)L(f_2)\).  We are given \(f_1,f_2\in\Contc(G^1)\) and
  define \(f\in\Contc(G^2)\) by
  \[
  f(g_1,g_2) = \conj{f_1(g_1)} f_2(g_1\cdot g_2).
  \]
  Let \(h'\in\Contc(G^2)\) be some function with
  \(h'\cdot f=f\), that is, \(h'\) is~\(1\) on the (compact) support
  of~\(f\).  We use~\eqref{eq:di_U-condition} and
  that~\(d_2^*(U)^*\) commutes with \(\Cont_0(G^2)\) to compute
  \[
  T_{h'}^* M_f d_0^*(U) T_{h'}
  = T_{h'}^* M_f d_2^*(U)^* d_1^*(U) T_{h'}
  = T_{h'}^* d_2^*(U)^* M_f d_1^*(U) T_{h'}.
  \]
  We are going to prove that
  \begin{align}
    \label{eq:Lf_star-rep_2}
    T_{h'}^* M_f d_0^*(U) T_{h'} &= L(f_1^* *f_2),\\
    \label{eq:Lf_star-rep_3}
    T_{h'}^* d_2^*(U)^* M_f d_1^*(U) T_{h'} &= L(f_1)^* L(f_2).
  \end{align}
  This will finish the proof of~\eqref{eq:Lf_star-rep}.

  We begin with some preparatory observations.  Our proof depends on
  the isomorphisms of \cstar{}correspondences in
  Figure~\ref{fig:corr_G2_G0} such as
  \begin{equation}
    \label{eq:sample_corr_G2_G0}
    \Lt^2(G^2,d_0,\lambda_0) \otimes_{\Cont_0(G^1)} \Lt^2(G^1,\rg,\alpha)
    \cong \Lt^2(G^2,v_1,\mu_1).
  \end{equation}
  This isomorphism is described in
  Lemma~\ref{lem:compose_measure_family} and maps
  \(\varphi_1\otimes\varphi_2\)
  for \(\varphi_1\in\Contc(G^2)\), \(\varphi_2\in\Contc(G^1)\) to
  the function
  \[
  (g_1,g_2)\mapsto \varphi_1(g_1,g_2)\cdot \varphi_2(d_0(g_1,g_2))
  = \varphi_1(g_1,g_2)\cdot \varphi_2(g_2).
  \]
  Hence the inverse isomorphism can be taken to
  send \(\varphi_1\mapsto \varphi_1\otimes k\), where
  \(k\in\Contc(G^1)\) is such that \(k(g_2)=1\) for all
  \((g_1,g_2)\in\supp\varphi_1\).  Similar remarks apply to all
  commutative triangles in Figure~\ref{fig:corr_G2_G0}.

  The operators \(T^*_{h'}\) and~\(T_{h'}\)
  in \eqref{eq:Lf_star-rep_2} and~\eqref{eq:Lf_star-rep_3}
  have slightly different meanings: the
  first treats \(h'\in \Lt^2(G^2,v_1,\mu_1)\), the second as \(h'\in
  \Lt^2(G^2,v_2,\mu_2)\); this is implicit in
  \eqref{eq:Lf_star-rep_2} and~\eqref{eq:Lf_star-rep_3}
  because \(M_f d_0^*(U)\) and \(d_2^*(U)^* M_f d_1^*(U)\) are
  operators from \(\Lt^2(G^2,v_2,\mu_2)\) to \(\Lt^2(G^2,v_1,\mu_1)\).
  We write \((T_{h'}^{v_1})^*\) or~\(T_{h'}^{v_2}\) to clarify in
  which \cstar{}correspondence we view~\(h'\).  We shall also
  need~\(T_{h'}^{d_0}\), and so on.  The definition of~\(d_i^*(U)\)
  through \(\Id_{\Lt^2(G^2,d_i,\lambda_i)} \otimes U\) implies
  \begin{equation}
    \label{eq:diU_T}
    d_i^*(U) T^{d_i}_{h'} = T^{d_i}_{h'} U\qquad
    \text{for }i=0,1,2.
  \end{equation}

  Now we compute the operator
  \((T^{v_1}_{h'})^* M_f d_0^*(U) T^{v_2}_{h'}\).
  First, we rewrite \(h'= h'\otimes k\),
  where~\(k\)
  is~\(1\)
  on a sufficiently large compact subset, compare the discussion
  after~\eqref{eq:sample_corr_G2_G0}.  Then
  \(T^{v_2}_{h'} = T^{d_0}_{h'} T^{\s}_k\)
  and \(T^{v_1}_{h'} = T^{d_0}_{h'} T^{\rg}_k\).
  Using~\eqref{eq:diU_T} as well, we get
  \[
  (T^{v_1}_{h'})^* M_f d_0^*(U) T^{v_2}_{h'}
  = (T^{\rg}_k)^* (T^{d_0}_{h'})^* M_f T^{d_0}_{h'} U T^{\s}_k.
  \]
  Let \(\varphi\in \Lt^2(G^1,\rg,\alpha)\), \(\xi\in\Hilm[F]\).
  The operator \((T^{d_0}_{h'})^* M_f T^{d_0}_{h'}\)
  does the following on \(\varphi\otimes\xi\):
  \[
  \varphi\otimes\xi \xrightarrow{T^{d_0}_{h'}}
  h'\otimes \varphi\otimes\xi \xrightarrow{M_f}
  f h'\otimes \varphi\otimes\xi \xrightarrow{(T^{d_0}_{h'})^*}
  \braket{h'}{f h'}_{\Lt^2(G^2,d_0,\lambda_0)}\cdot\varphi\otimes\xi,
  \]
  where the product in \(\braket{h'}{f h'}\cdot \varphi\)
  is the pointwise multiplication action of~\(\Cont_0(G^1)\)
  on \(\varphi\in \Lt^2(G^1,\rg,\alpha)\).
  Since \(h'=1\)
  on the support of~\(f\),
  we get \((T^{d_0}_{h'})^* M_f T^{d_0}_{h'} = M_{\lambda_0(f)}\) with
  \begin{multline*}
    \lambda_0(f)(g_2)
    = \braket{h'}{f h'}_{\Lt^2(G^2,d_0,\lambda_0)}(g_2)
    = \int_{G_{\rg(g_2)}} f(g_1,g_2) \dd\tilde{\alpha}_{\rg(g_2)}(g_1)
    \\= \int_{G^{\rg(g_2)}} f(g_1^{-1},g_2) \dd\alpha^{\rg(g_2)}(g_1)
    = \int_{G^{\rg(g_2)}} \conj{f_1(g_1^{-1})} f_2(g_1^{-1}\cdot g_2)
    \dd\alpha^{\rg(g_2)}(g_1)
    = (f_1^* * f_2)(g_2).
  \end{multline*}
  Plugging this into our computations above
  proves~\eqref{eq:Lf_star-rep_2}.

  Now we prove~\eqref{eq:Lf_star-rep_3}.  We rewrite \(T^{v_2}_{h'} =
  T^{d_1}_{h'} T^{\s}_k\) and \(T^{v_1}_{h'} = T^{d_2}_{h'}
  T^{\s}_k\) because \(v_2= \s\circ d_1\) and \(v_1 = \s\circ d_2\).
  We use~\eqref{eq:diU_T} to compute
  \[
  (T^{v_1}_{h'})^* d_2^*(U)^* M_f d_1^*(U) T_{h'}^{v_2}
  = (T^{\s}_k)^* U^* (T^{d_2}_{h'})^* M_f T^{d_1}_{h'} U T^{\s}_k.
  \]
  Let \(\varphi\in\Contc(G^1) \subseteq \Lt^2(G^1,\rg,\alpha)\) and
  \(\xi\in\Hilm[F]\).  Then~\(M_f T^{d_1}_{h'}\) first maps
  \(\varphi\otimes\xi\) to \(f\cdot h'\otimes \varphi\otimes \xi \in
  \Lt^2(G^2,d_1,\lambda_1) \otimes_{\Cont_0(G^1)}
  \Lt^2(G^1,\rg,\alpha)\otimes_\varphi\Hilm[F]\).  To apply the
  operator~\((T^{d_2}_{h'})^*\) to this, we apply the canonical
  isomorphism
  \[
  \Lt^2(G^2,d_1,\lambda_1) \otimes_{\Cont_0(G^1)}
  \Lt^2(G^1,\rg,\alpha)
  \cong \Lt^2(G^2,v_0,\mu_0) \cong
  \Lt^2(G^2,d_2,\lambda_2) \otimes_{\Cont_0(G^1)}
  \Lt^2(G^1,\rg,\alpha)
  \]
  from Figure~\ref{fig:corr_G2_G0} and then take the inner product
  with~\(h'\) in the first tensor factor.  Since~\(h'\) is~\(1\)
  wherever our functions are supported, we have \(f h' = f\), and
  the inner product with~\(h'\) simply applies the integration
  map~\(\lambda_2\).  Thus~\((T^{d_2}_{h'})^* M_f T^{d_1}_{h'}\)
  maps~\(\varphi\otimes\xi\) to~\(\psi\otimes\xi\) with
  \begin{align*}
    \psi(g_1)
    &= \int_{G^{\s(g_1)}} f(g_1,g_2) (\varphi\circ d_1)(g_1,g_2)
    \dd\alpha^{\s(g_1)}(g_2)
    \\&= \int_{G^{\s(g_1)}} \conj{f_1(g_1)} f_2(g_1 g_2) \varphi(g_1 g_2)
    \dd\alpha^{\s(g_1)}(g_2)
    \\&= \conj{f_1(g_1)} \int_{G^{\s(g_1)}} f_2(g_2) \varphi(g_2)
    \dd\alpha^{\rg(g_1)}(g_2).
  \end{align*}
  As a consequence,
  \[
  (T^{d_2}_{h'})^* M_f T^{d_1}_{h'}
  = M^*_{f_1} T^{\rg}_k (T^{\rg}_k)^* M_{f_2}.
  \]
  Putting things together gives
\[
  (T^{v_1}_{h'})^* (M_f d_2^*(U)^* d_1^*(U)) T_{h'}^{v_2}\\
  = (T^{\s}_k)^* U^* M^*_{f_1} T^{\rg}_k
  (T^{\rg}_k)^* M_{f_2} U T^{\s}_k
  = L(f_1)^**L(f_2).
\]
\end{proof}

Since the \Star{}representation~\(L\) of~\((\Contc(G^1),*)\) is
bounded in the \(I\)\nb-norm, it extends uniquely to a
representation of~\(\Cst(G,\alpha)\), which we still denote
by~\(L\).  This is the integrated form of the
representation~\((\varphi,U)\).

\begin{example}[(The integrated regular representation)]
  We describe the integrated form of the regular representation
  of~\((G,\alpha)\) introduced in Section~\ref{sec:regular-rep}.
  We continue to use the notation from that construction.  The
  underlying \(\Cont_0(G^0),\Cont_0(G^0)\)-correspondence is
  \((\Hilm[F],\varphi)\defeq \rg^*\Lt^2(G^1,\s,\tilde\alpha)\).
  The isomorphism
  \[
  U\colon \Lt^2(G^1\times_{\s,\rg}G^1,\nu_2,\mu_2)\congto
  \Lt^2(G^1\times_{\rg,\rg}G^1,\s\circ\pr_2,\tilde\alpha\circ\alpha)
  \]
  is the unique extension of the map on continuous functions with
  compact support defined by \(U(\zeta)(g,h)\defeq
  \zeta(g,g^{-1}h)\) for \(\zeta\in
  \Contc(G^1\times_{\s,\rg}G^1)\) and \((g,h)\in
  G^1\times_{\rg,\rg}G^1\).
  Equation~\eqref{eq:integrated_form_explicit} describes \(L\colon
  \Contc(G^1)\to \Bound(\Hilm[F])\) by \(L(f)=T_{f_1}^* U T_{f_2}\)
  if \(f=\conj{f_1}\cdot f_2\) with \(f_1,f_2\in \Contc(G^1)\).
  In particular, we may choose \(k\in \Contc(G^1)\) with \(k=1\)
  on \(\supp(f)\) and compute \(L(f)\) as \(T_k^* U T_f\).  With
  our previous identifications, \(T_k^*\) is the operator from
  \(\Lt^2(G^1\times_{\rg,\rg}G^1,\s\circ\pr_2,\tilde\alpha\circ\alpha)\)
  to \(\Hilm[F]=\Lt^2(G^1,\s,\tilde\alpha)\) given by
  \[
  T_k^*(\zeta)(h)=\int_G \conj{k(g)}\zeta(g,h)\dd\alpha^{\rg(h)}(g)
  \]
  for all \(\zeta\in \Contc(G^1\times_{\rg,\rg}G^1)\) and \(h\in G^1\).
  Hence
  \begin{multline*}
    L(f)\xi(h)
    = T_k^*U(f\otimes \xi)(h)
    = \int_G\conj{k(g)}U(f\otimes\xi)(g,h)\dd\alpha^{\rg(h)}(g)
    \\= \int_G \conj{k(g)}f(g)\xi(g^{-1}h)\dd\alpha^{\rg(h)}(g)
    = \int_G f(g)\xi(g^{-1}h)\dd\alpha^{\rg(h)}(g)
    = (f*\xi)(h)
  \end{multline*}
  for all \(f\in \Contc(G^1)\subseteq \Cst(G,\alpha)\), \(\xi\in
  \Contc(G^1)\subseteq \Hilm[F]\) and \(h\in G^1\).  So \(L(f)\) is the
  operator of left convolution with~\(f\).  This is the usual
  definition of the regular representation.
\end{example}

\section{Disintegration}
\label{sec:disintegrate}

In this section, we construct a representation~\((\varphi,U)\)
of~\((G,\alpha)\) from a representation of~\(\Cst(G,\alpha)\).
Actually, we shall start with a more technical setting, allowing
densely defined representations of~\(\Contc(G^1)\).  Renault's
Disintegration Theorem also applies in this generality, and several
results need such representations.

\begin{definition}
  \label{def:prerepresentation}
  Let~\(\Hilm[F]\)
  be a Hilbert module over a \cstar{}algebra~\(D\)
  and let~\(\Hilm[F]_0\)
  be a vector space with a linear map
  \(\iota\colon \Hilm[F]_0\to\Hilm[F]\)
  with dense image.  Let \(\Hom(\Hilm[F]_0,\Hilm[F])\),
  be the vector space of all linear maps \(\Hilm[F]_0\to \Hilm[F]\).
  A \emph{pre-\alb{}representation} of~\(\Contc(G^1)\)
  on~\(\Hilm[F]\)
  is a linear map \(L\colon \Contc(G^1)\to \Hom(\Hilm[F]_0,\Hilm[F])\)
  such that
  \begin{enumerate}
  \item for all \(\xi, \eta\in \Hilm[F]_0\),
    the map \(f\mapsto \braket{\iota(\xi)}{L(f)\eta}_D\)
    from~\(\Contc(G^1)\)
    to~\(D\)
    is continuous in the inductive limit topology on~\(\Contc(G^1)\)
    and the norm topology on \(D\);
  \item
    \(\braket{L(f_1)\xi}{L(f_2)\eta}_D =
    \braket{\iota(\xi)}{L(f_1^**f_2)\eta}_D\)
    for all \(\xi,\eta\in \Hilm[F]_0\);
  \item the linear span of~\(L(f)\xi\)
    for \(f\in\Contc(G^1)\),
    \(\xi\in \Hilm[F]_0\) is dense in~\(\Hilm[F]\).
  \end{enumerate}
\end{definition}

We do not need the map~\(\iota\)
to be injective.  The assumptions in
Definition~\ref{def:prerepresentation} imply, however, that
\(L(f)\xi=0\)
for all \(f\in\Contc(G^1)\),
\(\xi\in\Hilm[F]_0\)
with \(\iota(\xi)=0\)
because \(\braket{L(f)\xi}{L(f_2)\xi_2}=0\)
for all \(f_2\in\Contc(G^1)\),
\(\xi_2\in\Hilm[F]_0\)
by~(2) and linear combinations of~\(L(f_2)\xi_2\)
are dense in~\(\Hilm[F]\)
by~(3).  So it would be no loss of generality to assume~\(\iota\)
to be injective: we may replace \(\iota\colon \Hilm[F]_0\to \Hilm[F]\)
by the injective linear map
\(\tilde\iota\colon \tilde\Hilm[F]_0\defeq \Hilm[F]_0/\ker(\iota)\to\Hilm[F]\),
\(\tilde\iota(\xi+\ker(\iota))\defeq \iota(\xi)\),
and \(L\)
by \(\tilde L\colon \Contc(G^1)\to \Hom(\tilde\Hilm[F]_0,\Hilm[F])\),
\(\tilde L(f)(\xi+\ker(\iota))\defeq L(f)\xi\).

Throughout this section, we fix a groupoid~\(G\)
with a Haar system~\(\alpha\)
and a pre-representation \((L,\Hilm[F]_0,\iota)\)
of~\(\Contc(G^1)\).
Disintegration will produce a representation~\((\varphi,U)\)
of~\((G,\alpha)\)
on~\(\Hilm[F]\).  First, we
construct the representation~\(\varphi\)
of~\(\Cont_0(G^0)\)
on~\(\Hilm[F]\).  Define
\(\rg^*(f_0)\cdot f_1\in\Contc(G^1)\)
for \(f_0\in\Contb(G^0)\),
\(f_1\in\Contc(G^1)\)
by \(\rg^*(f_0)\cdot f_1(g) \defeq f_0(\rg(g))\cdot f_1(g)\)
for all \(g\in G^1\).  Then define
\[
\varphi_0\colon
\Contb(G^0) \odot \Contc(G^1) \odot \Hilm[F]_0 \to \Hilm[F],\qquad
f_0\otimes f_1 \otimes \xi\mapsto L(\rg^*(f_0)\cdot f_1)\xi.
\]

\begin{lemma}
  \label{lem:disintegrate_varphi}
  There is a unique representation \(\varphi
\colon
  \Cont_0(G^0)\to\Bound(\Hilm[F])\) with
  \[
  \varphi(f_0)(L(f_1)\xi) = \varphi_0(f_0\otimes f_1\otimes\xi)
  \]
  for all \(f_0\in\Cont_0(G^0)\), \(f_1\in\Contc(G^1)\),
  \(\xi\in\Hilm[F]_0\).
\end{lemma}

\begin{proof}
  Let \(f_0\in\Contb(G^1)\).
  Define \(f_0'\in\Contb(G^1)\)
  by
  \[
  f_0'(g) \defeq \sqrt{\norm{f_0}_\infty^2 - \abs{f_0(g)}^2}
  \]
  for all \(g\in G^1\),
  so that \(f_0^* f_0 + (f_0')^* f_0' = \norm{f_0}_\infty^2\cdot 1\)
  in~\(\Contb(G^1)\).
  If \(f_1\in\Contc(G^1)\), \(\xi\in\Hilm[F]_0\), then
  \begin{align*}
    &\phantom{{}\le{}}
    \braket{\varphi_0(f_0\otimes f_1\otimes\xi)}
    {\varphi_0(f_0\otimes f_1\otimes\xi)}
    \\&\le
    \braket{\varphi_0(f_0\otimes f_1\otimes\xi)}
    {\varphi_0(f_0\otimes f_1\otimes\xi)}
    + \braket{\varphi_0(f'_0\otimes f_1\otimes\xi)}
    {\varphi_0(f'_0\otimes f_1\otimes\xi)}
    \\&=
    \braket{L(\rg^*(f_0)f_1)\xi} {L(\rg^*(f_0) f_1)\xi}
    + \braket{L(\rg^*(f_0')f_1)\xi} {L(\rg^*(f_0') f_1)\xi}
    \\&= \braket{\iota(\xi)}
    {L(f_1^* * (\rg^*(f_0^* f_0 + (f_0')^* f_0') f_1))\xi}
    \\&= \norm{f_0}^2 \braket{\iota(\xi)} {L(f_1^* * f_1)\xi}
    = \norm{f_0}^2 \braket{L(f_1)\xi} {L(f_1)\xi}.
  \end{align*}
  Since \(L(\Contc(G^1))(\Hilm[F]_0)\)
  is linearly dense in~\(\Hilm[F]\),
  there is a unique bounded linear operator
  \(\varphi(f_0)\colon \Hilm[F]\to\Hilm[F]\) with
  \[
  \varphi(f_0)(L(f_1)\xi)
  = \varphi_0(f_0\otimes f_1\otimes\xi)
  = L(\rg^*(f_0)f_1) \xi
  \]
  for all \(f_1\in\Contc(G^1)\), \(\xi\in\Hilm[F]_0\).
  The operator \(\varphi(f_0)\) is adjointable with
  adjoint~\(\varphi(f_0^*)\) because \(f_1^* * (\rg^*(f_0)f_1') =
  (\rg^*(f_0^*)f_1)^* * f_1'\) in~\(\Contc(G^1)\) for all
  \(f_1,f_1'\in\Contc(G^1)\).  The map \(\varphi\colon
  \Contb(G^1)\to\Bound(\Hilm[F])\) is linear and multiplicative
  because the action of~\(\Contb(G^1)\) on~\(\Contc(G^1)\) by
  pointwise multiplication is a module structure.  The restriction
  of~\(\varphi\) to~\(\Cont_0(G^1)\) is nondegenerate because
  \(\Cont_0(G^1)\cdot \Contc(G^1) = \Contc(G^1)\) and
  \(L(\Contc(G^1))(\Hilm[F]_0)\) is linearly dense in~\(\Hilm[F]\).
\end{proof}

Next, we are going to construct linear maps with dense range
\begin{align*}
  \tau_\s&\colon\Contc(G^2)\odot \Hilm[F]_0 \to
  \Lt^2(G^1,\s,\tilde\alpha) \otimes_\varphi \Hilm[F],\\
  \tau_\rg&\colon\Contc(G^2)\odot \Hilm[F]_0 \to
  \Lt^2(G^1,\rg,\alpha) \otimes_\varphi \Hilm[F],
\end{align*}
such that
\[
\braket{\tau_\s(x)}{\tau_\s(x)} =
\braket{\tau_\rg(x)}{\tau_\rg(x)}
\]
for all \(x\in \Contc(G^2)\odot\Hilm[F]_0\).  Hence there is a
unique unitary operator
\[
U\colon \Lt^2(G^1,\s,\tilde\alpha) \otimes_\varphi \Hilm[F] \congto
\Lt^2(G^1,\rg,\alpha) \otimes_\varphi \Hilm[F]
\]
with \(U\bigl(\tau_\s(x)\bigr) = \tau_\rg(x)\) for all
\(x\in\Contc(G^2)\odot \Hilm[F]_0\).  We will check later
that~\((\varphi,U)\) is a representation of~\((G,\alpha)\).
The construction of \(\tau_\s,\tau_\rg\) is a special
case of the following lemma.

\begin{lemma}
  \label{lem:dense_in_HilmF}
  Let \(p\colon X\to G^0\) be a continuous map and let~\(\lambda\)
  be a continuous family of measures along~\(p\).  The map
  \[
  \tau\colon \Contc(X) \odot \Contc(G^1) \odot \Hilm[F]_0 \to
  \Lt^2(X,p,\lambda) \otimes_\varphi \Hilm[F],\qquad
  f_0\otimes f_1\otimes \xi \mapsto f_0\otimes L(f_1)\xi,
  \]
  extends uniquely to a linear map
  \[
  \bar\tau\colon \Contc(X\times_{p,G^0,\rg} G^1) \odot
  \Hilm[F]_0 \to \Lt^2(X,p,\lambda) \otimes_\varphi \Hilm[F],
  \]
  such that \(f\mapsto \bar\tau(f\otimes \xi)\)
  for \(f\in\Contc(X\times_{p,G^0,\rg} G^1)\)
  is continuous in the inductive limit topology for all
  \(\xi\in\Hilm[F]_0\).
  If \(F_1,F_2\in\Contc(X\times_{p,G^0,\rg} G^1)\),
  \(\xi_1,\xi_2\in\Hilm[F]_0\), then
  \begin{equation}
    \label{eq:innprod_dense_in_HilmF0}
    \braket{\bar\tau(F_1\otimes \xi_1)}{\bar\tau(F_2\otimes \xi_2)}
    = \braket{\iota(\xi_1)}{L(\braket{F_1}{F_2})\xi_2},
  \end{equation}
  where \(\braket{F_1}{F_2}\in\Contc(G^1)\) is defined by
  \begin{align}
    \label{eq:innprod_dense_in_HilmF}
    \braket{F_1}{F_2}(g) &\defeq \iint \conj{F_1(x,h^{-1})}
    F_2(x,h^{-1} g) \dd\lambda^{\s(h)}(x)\dd\alpha^{\rg(g)}(h)
    \\ \notag &= \iint \conj{F_1(x,h)}
    F_2(x,h g) \dd\lambda^{\rg(h)}(x) \dd\tilde\alpha_{\rg(g)}(h).
  \end{align}
\end{lemma}
\begin{proof}
  Let \(f_1,f_3\in\Contc(X)\), \(f_2,f_4\in\Contc(G^1)\),
  \(\xi_1,\xi_2\in \Hilm[F]_0\).  We compute
  \begin{align*}
    \braket{\tau(f_1\otimes f_2\otimes \xi_1)}
    {\tau(f_3\otimes f_4\otimes \xi_2)}
    &= \braket{L(f_2)\xi_1}{\varphi(\braket{f_1}{f_3}) L(f_4)\xi_2}
    \\&= \braket{\iota(\xi_1)}{L(f_2^* * (\rg^*(\braket{f_1}{f_3}) f_4))\xi_2}.
  \end{align*}
  Here \(f_2^* * (\rg^*(\braket{f_1}{f_3}) f_4) \in \Contc(G^1)\)
  is given by
  \begin{align*}
    f_2^* * (\rg^*(\braket{f_1}{f_3}) f_4)(g)
    &= \int_{G^1} f_2^*(h) (\rg^*(\braket{f_1}{f_3}) f_4)(h^{-1} g)
    \dd\alpha^{\rg(g)}(h)
    \\&= \int_{G^1} \int_X \conj{f_2(h^{-1})} \conj{f_1(x)} f_3(x) f_4(h^{-1} g)
    \dd\lambda^{\rg(h^{-1} g)}(x) \dd\alpha^{\rg(g)}(h)
    \\&= \int_{G^1} \int_X \conj{f_2(h)} \conj{f_1(x)} f_3(x) f_4(h g)
    \dd\lambda^{\rg(h)}(x) \dd\tilde\alpha_{\rg(g)}(h).
  \end{align*}
  This is the right hand side in~\eqref{eq:innprod_dense_in_HilmF} if
  we let \(F_1(x,g) \defeq f_1(x) f_2(g)\)
  and \(F_2(x,g) \defeq f_3(x) f_4(g)\).
  Thus~\eqref{eq:innprod_dense_in_HilmF0} with \(\bar\tau=\tau\)
  holds for \(F_1,F_2\)
  in the image of \(\Contc(X)\odot \Contc(G^1)\)
  in \(\Contc(X\times_{p,G^0,\rg} G^1)\).

  Let \(F\in\Contc(X\times_{p,G^0,\rg} G^1)\) and
  \(\xi\in\Hilm[F]_0\).  Let \(V\) and~\(W\) be open, relatively compact
  subsets of~\(X\) and~\(G^1\), respectively, so that
  \(V\times_{p,G^0,\rg} W\) is a neighbourhood of the support
  of~\(F\), that is, \(F\in\Cont_0(V \times_{p,G^0,\rg} W)\).  The
  linear span of functions of the form \(f_1\otimes f_2\) with
  \(f_1\in\Cont_0(V) \subseteq \Contc(X)\), \(f_2\in\Cont_0(W)
  \subseteq \Contc(G^1)\) is dense in the Banach space
  \(\Cont_0(V\times_{p,G^0,\rg} W)\subseteq
  \Contc(X\times_{p,G^0,\rg} G^1)\) by the Stone--Weierstraß
  Theorem.  Hence there is a sequence~\((F_n)_{n\in\N}\) in
  \(\Cont_0(V)\odot \Cont_0(W)\) whose image in
  \(\Cont_0(V\times_{p,G^0,\rg} W)\) converges towards~\(F\).  We
  claim that~\(\tau(F_n\otimes\xi)\) is a Cauchy sequence
  in~\(\Lt^2(X,p,\lambda) \otimes_\varphi \Hilm[F]\).  Then we are
  going to define
  \(\bar\tau(F\otimes\xi) = \lim \tau(F_n\otimes\xi)\).

  We may use~\eqref{eq:innprod_dense_in_HilmF0} to compute
  \(\braket{\tau(F_n\otimes\xi)}{\tau(F_m\otimes\xi)}\) for all
  \(n,m\in\N\) because it holds for elementary tensors in the place
  of \(F_n\) and~\(F_m\).
  The continuity assumption for~\(L\) in
  Definition~\ref{def:prerepresentation} shows that this is a
  continuous bilinear map in the two entries \(F_n,F_m\).
  Therefore,
  \[
  \lim_{n,m\to\infty}
  \braket{\tau(F_n\otimes\xi)}{\tau(F_m\otimes\xi)}
  = \braket{\iota(\xi)}{L(\braket{F}{F})\xi}
  \]
  with \(\braket{F}{F}\in\Contc(G^1)\) as
  in~\eqref{eq:innprod_dense_in_HilmF}.  Hence
  \begin{multline*}
    \norm{\tau(F_n\otimes\xi) - \tau(F_m\otimes\xi)}^2
    =
    \braket{\tau(F_n\otimes\xi)}{\tau(F_n\otimes\xi)} -
    \braket{\tau(F_n\otimes\xi)}{\tau(F_m\otimes\xi)} \\-
    \braket{\tau(F_m\otimes\xi)}{\tau(F_n\otimes\xi)} +
    \braket{\tau(F_m\otimes\xi)}{\tau(F_m\otimes\xi)}
  \end{multline*}
  converges to~\(0\)
  for \(n,m\to\infty\).
  Thus \(\tau(F_n\otimes\xi)\)
  is a Cauchy sequence in
  \(\Lt^2(X,p,\lambda) \otimes_\varphi \Hilm[F]\).
  We let~\(\bar\tau(F\otimes\xi)\)
  be its limit.  This does not depend on the choice of \(V,W\)
  and~\((F_n)\)
  because mixing two sequences of this type gives a Cauchy sequence as
  well, by the same argument.  The map
  \((F,\xi)\mapsto \bar\tau(F\otimes\xi)\)
  is bilinear and hence extends to a linear map
  \(\bar\tau\colon \Contc(X\times_{p,G^0,\rg} G^1) \odot \Hilm[F]_0
  \to \Lt^2(X,p,\lambda)\otimes_\varphi \Hilm[F]\).
  It has dense image because already~\(\tau\)
  has dense image.  The formula~\eqref{eq:innprod_dense_in_HilmF0}
  holds everywhere by continuity and because it holds for~\(F_1,F_2\)
  in the dense subspace \(\Contc(X)\odot \Contc(G^1)\)
  of \(\Contc(X\times_{p,G^0,\rg} G^1)\).
\end{proof}

We apply Lemma~\ref{lem:dense_in_HilmF} to the maps \(\rg,s\colon
G^1 \rightrightarrows G^0\) with the measure families
\(\alpha,\tilde\alpha\).  This gives linear maps with dense image
\begin{align*}
  \tau_\s&\colon \Contc(G^1 \times_{\s,G^0,\rg} G^1) \odot
  \Hilm[F]_0 \to
  \Lt^2(G^1,\s,\tilde\alpha) \otimes_\varphi \Hilm[F],\\
  \tau_\rg&\colon \Contc(G^1 \times_{\rg,G^0,\rg} G^1) \odot
  \Hilm[F]_0 \to
  \Lt^2(G^1,\rg,\alpha) \otimes_\varphi \Hilm[F].
\end{align*}
We identify \(G^1 \times_{\s,G^0,\rg} G^1 = G^2\) with \(G^1
\times_{\rg,G^0,\rg} G^1\) through the homeomorphisms \((g,h)\mapsto
(g, g^{\pm 1} h)\) going back and forth.  Let \(F_1,F_2\in\Contc(G^2)\),
\(\xi_1,\xi_2\in\Hilm[F]_0\).  Then
\begin{align*}
  \braket{\tau_\s(F_1\otimes \xi_1)}
  {\tau_\s(F_2\otimes \xi_2)}
  &= \braket{\iota(\xi_1)} {L(\braket{F_1}{F_2}_\s)\xi_2},\\
  \braket{\tau_\rg(F_1\otimes \xi_1)}
  {\tau_\rg(F_2\otimes \xi_2)}
  &= \braket{\iota(\xi_1)} {L(\braket{F_1}{F_2}_\rg)\xi_2}
\end{align*}
with
\begin{align*}
  \braket{F_1}{F_2}_\s(k) &=
  \iint \conj{F_1(x,h,x h)} F_2(x,h k,x h k) \dd\tilde\alpha_{\rg(h)}(x)
  \dd\tilde\alpha_{\rg(k)}(h),\\
  \braket{F_1}{F_2}_\rg(k) &=
  \iint \conj{F_1(x,x^{-1} h, h)} F_2(x,x^{-1} h k, h k)
  \dd\alpha^{\rg(h)}(x) \dd\tilde\alpha_{\rg(k)}(h)
\end{align*}
by~\eqref{eq:innprod_dense_in_HilmF}.  Here we described points
in~\(G^2\) through \(x,h,y\in G^1\) with \(x h = y\) to
clarify the identification of \(G^1\times_{\rg,G^0,\rg} G^1\)
with~\(G^2\).  Equation~\eqref{eq:compare_integrals_rep} applied to
the function
\[
f^k(g,h) \defeq \conj{F_1(g,h)} F_2(g,h k)
\]
in \(\Contc(G^1\times_{\s,G^0,\rg} G_{\rg(k)})\) for fixed \(k\in
G^1\) implies \(\braket{F_1}{F_2}_\s(k) = \braket{F_1}{F_2}_\rg(k)\)
for all \(k\in G^1\).  Hence \(\tau_\s\) and~\(\tau_\rg\) induce the
same inner product on \(\Contc(G^2)\odot \Hilm[F]_0\).  So there is
a unique unitary
\[
U\colon \Lt^2(G^1,\s,\tilde\alpha) \otimes_\varphi \Hilm[F] \congto
\Lt^2(G^1,\rg,\alpha) \otimes_\varphi \Hilm[F]
\]
with \(U\tau_\s(x) = \tau_\rg(x)\) for all \(x\in\Contc(G^2)\otimes
\Hilm[F]_0\).
More precisely, \(U\tau_\s(x) = \tau_\rg(\upsilon(x))\), where
\(\upsilon\colon \Contc(G^1\times_{\s,\rg}G^1)\odot\Hilm[F]_0\congto
\Contc(G^1\times_{\rg,\rg}G^1)\odot\Hilm[F]_0\) is the isomorphism
induced by the homeomorphism \(\Upsilon\colon
G^1\times_{\s,\rg}G^1\congto G^1\times_{\rg,\rg}G^1\),
\((g,h)\mapsto (g,gh)\) in~\eqref{eq:fundamental-homeo}.

\begin{proposition}
  \label{pro:disintegrate}
  The pair~\((\varphi,U)\)
  associated to a pre-representation~\(L\) of~\(\Contc(G^1)\)
  is a representation of~\((G,\alpha)\).
\end{proposition}

\begin{proof}
  First we check that~\(U\) is an isomorphism of correspondences.
  That is, it is a \(\Cont_0(G^1)\)-module homomorphism for the
  canonical left \(\Cont_0(G^1)\)-module structures by pointwise
  multiplication on the first tensor factors
  \(\Lt^2(G^1,\s,\tilde\alpha)\) and \(\Lt^2(G^1,\rg,\alpha)\).  The
  maps \(\tau_\s\) and~\(\tau_\rg\) are \(\Cont_0(G^1)\)-module
  homomorphism if we let \(f_1\in \Cont_0(G^1)\) act on \(f_2\in
  \Contc(G^1\times_{\s,\rg}G^1)\odot\Hilm[F]_0\) and \(f_2'\in
  \Contc(G^1\times_{\rg,\rg}G^1)\odot\Hilm[F]_0\) by
  \[
  (f_1\cdot f_2)(g,h) \defeq f_1(g)f_2(g,h),\qquad
  (f_1\cdot f_2')(g,k) \defeq f_1(g)f_2'(g,k).
  \]
  These give the same \(\Cont_0(G^1)\)-module structure on
  \(\Contc(G^2)\odot\Hilm[F]_0\).  So the unitary~\(U\) is a
  \(\Cont_0(G^1)\)-module homomorphism.

  The unitaries~\(d_i^*(U)\) in \eqref{eq:d1_U}--\eqref{eq:d2_U} act
  between \(\Lt^2(G^2,v_j,\mu_j) \otimes_\varphi\Hilm[F]\) for
  \(j\in\{0,1,2\}\setminus\{i\}\).  Lemma~\ref{lem:dense_in_HilmF}
  provides a linear map with dense range
  \[
  \tau_{v_j}\colon
  \Contc(G^2\times_{v_j,G^0,\rg} G^1)\odot\Hilm[F]_0 \to
  \Lt^2(G^2,v_j,\mu_j) \otimes_\varphi\Hilm[F].
  \]
  The elements of \(G^2\times_{v_j,G^0,\rg} G^1\) are configurations
  of three arrows \((g,h,x)\) with \((g,h)\in G^2\) and
  \(\rg(x)=\rg(g)\) for \(j=0\), \(\rg(x)=\s(g)=\rg(h)\) for \(j=1\),
  and \(\rg(x)=\s(h)\) for \(j=2\), respectively.  In particular,
  \(G^2\times_{v_2,G^0,\rg} G^1\) is the space~\(G^3\) of triples of
  composable arrows~\((g,h,x)\) in~\(G\).

  To define~\(d_0^*(U)\) as in~\eqref{eq:d0_U}, we identify \(G^2
  \times_{v_2,G^0,\rg} G^1 \cong G^1 \times_{\s,G^0,\rg\circ\pr_1}
  (G^1\times_{\s,\rg} G^1)\) and \(G^2 \times_{v_1,G^0,\rg} G^1
  \cong G^1 \times_{\s,G^0,\rg\circ\pr_1} (G^1\times_{\rg,\rg}
  G^1)\).  Since~\(U\) is induced by the homeomorphism
  \(G^1\times_{\rg,\rg} G^1 \to G^1\times_{\s,\rg} G^1\),
  \((h,k)\mapsto (h,h^{-1} k)\), the operator \(d_0^*(U)\) is the
  unique extension of
  \[
  \Contc(G^2\times_{v_2,G^0,\rg} G^1)\odot\Hilm[F]_0 \to
  \Contc(G^2\times_{v_1,G^0,\rg} G^1)\odot\Hilm[F]_0,\qquad
  f\otimes \xi\mapsto (f\circ \Upsilon_0)\otimes \xi,
  \]
  with \(\Upsilon_0(g,h,x) \defeq (g, h, h^{-1} x)\).  Similarly,
  \(d_2^*(U)\) is the unique extension of
  \[
  \Contc(G^2\times_{v_1,G^0,\rg} G^1)\odot\Hilm[F]_0 \to
  \Contc(G^2\times_{v_0,G^0,\rg} G^1)\odot\Hilm[F]_0,\qquad
  f\otimes \xi\mapsto (f\circ \Upsilon_2)\otimes \xi,
  \]
  with \(\Upsilon_2(g,h,x) \defeq (g, h, g^{-1} x)\).  And~\(d_1^*(U)\)
  is the unique extension of
  \[
  \Contc(G^2\times_{v_2,G^0,\rg} G^1)\odot\Hilm[F]_0 \to
  \Contc(G^2\times_{v_0,G^0,\rg} G^1)\odot\Hilm[F]_0,\qquad
  f\otimes \xi\mapsto (f\circ \Upsilon_1)\otimes \xi,
  \]
  with \(\Upsilon_1(g,h,x) \defeq (g, h, (g h)^{-1} x)\).  The
  equation \(d_2^*(U) d_0^*(U) = d_1^*(U)\) follows from \(\Upsilon_1
  = \Upsilon_0 \circ \Upsilon_2\), which is the associativity of the
  multiplication in~\(G\).
\end{proof}

\section{Integration versus disintegration}
\label{sec:back_and_forth}

This section finishes the proof of
Theorem~\ref{the:universal_groupoid_Haus}.  Let~\(\Hilm[F]\) be a
Hilbert module over a \cstar{}algebra~\(D\).  Given a
representation of~\((G,\alpha)\) on~\(\Hilm[F]\), we have
constructed a representation of the convolution algebra
\(\Contc(G^1)\) bounded with respect to the \(I\)\nb-norm in
Section~\ref{sec:integrate}; this extends to a representation
of~\(\Cst(G,\alpha)\).  Conversely, given a representation
of~\(\Cst(G,\alpha)\) or merely a pre-representation of
\(\Contc(G^1)\) as in Definition~\ref{def:prerepresentation},
we have constructed a representation of~\((G,\alpha)\) in
Section~\ref{sec:disintegrate}.  First we prove that these two
constructions are inverse to each other.  Hence we get the asserted
bijection between representations of \((G,\alpha)\)
and~\(\Cst(G,\alpha)\).  Then we check that the bijection has the
two naturality properties in
Theorem~\ref{the:universal_groupoid_Haus}.

\begin{proposition}
  \label{pro:back_and_forth_from_L}
  Let~\((L,\Hilm[F]_0,\iota)\) be a pre-representation
  of~\(\Contc(G^1)\) on~\(\Hilm[F]\).  Let~\((\varphi,U)\) be its
  disintegration.  Then the integrated form~\(L'\) of~\((\varphi,U)\)
  satisfies \(L'(f)(\iota(\xi)) = L(f)(\xi)\) for all
  \(f\in\Contc(G^1)\), \(\xi\in\Hilm[F]_0\), and
 \(L'(f)(L(f_2)\xi) = L(f*f_2)(\xi)\) for all
 \(f,f_2\in\Contc(G^1)\), \(\xi\in\Hilm[F]_0\).
\end{proposition}

\begin{proof}
  Let \(f_1,f_2,f_3\in\Contc(G^1)\) and \(\xi_1,\xi_2\in\Hilm[F]_0\).
  We compute the inner product
  \(\braket{f_3\otimes\iota(\xi_1)}{f_1\otimes L(f_2)\xi_2}\) in
  \(\Lt^2(G^1,\rg,\alpha)\otimes_\varphi \Hilm[F]\):
  \[
  \braket{f_3\otimes\iota(\xi_1)}{f_1\otimes L(f_2)\xi_2}
  = \braket{\iota(\xi_1)}{\varphi(\braket{f_3}{f_1}) L(f_2)\xi_2}
  = \braket{\iota(\xi_1)}{L(\rg^*(\braket{f_3}{f_1})\cdot f_2)\xi_2},
  \]
  where \(\rg^*(\braket{f_3}{f_1})\cdot f_2\in\Contc(G^1)\) is given by
  \[
  \rg^*(\braket{f_3}{f_1})\cdot f_2(g)
  = \int_{G^1} \conj{f_3(x)} f_1(x) f_2(g) \dd\alpha^{\rg(g)}(x).
  \]
  Thus \(T_{f_3}^*(f_1\otimes L(f_2)\xi_2) =
  L(\rg^*(\braket{f_3}{f_1})\cdot f_2)\xi_2\); the annihilation
  operator~\(T_{f_3}^*\) is the adjoint of the creation
  operator~\(T_{f_3}\) as in Definition~\ref{def:creation}.
  Define \(F\in\Contc(G^1\times_{\rg,G^0,\rg}
  G^1)\) by \(F(x,g) \defeq f_1(x) f_2(x^{-1} g)\).
  The unitary~\(U\) maps \(f_1 \otimes L(f_2)\xi_2\)
  to \(\tau(F\otimes\xi_2)\).
  Since the image of \(\Contc(G^1)\odot \Contc(G^1)\) in \(\Contc(G^2)\)
  is dense in the inductive limit topology, the computation above
  implies that \(T_{f_3}^* U(f_1 \otimes L(f_2)\xi_2) =
  L(\psi)\xi_2\) with
  \[
  \psi = \int \conj{f_3(x)} F(x,g) \dd\alpha^{\rg(g)}(x)
  = \int \conj{f_3(x)} f_1(x) f_2(x^{-1} g) \dd\alpha^{\rg(g)}(x)
  = (\conj{f_3} f_1) * f_2(g).
  \]
  Summing up,
  \[
  T_{f_3}^* U T_{f_1} (L(f_2)\xi)
  = T_{f_3}^* U(f_1 \otimes L(f_2)\xi)
  = L\bigl(\bigl(\conj{f_3}f_1\bigr) * f_2\bigr)\xi
  \]
  for all \(f_1,f_2,f_3\in\Contc(G^1)\), \(\xi\in\Hilm[F]_0\).

  The integrated form~\(L'\) of~\((\varphi,U)\) is
  \(L'\bigl(\conj{f_3}f_1\bigr) \defeq T_{f_3}^* U T_{f_1}\)
  by~\eqref{eq:integrated_form_explicit}.  So the computation above
  shows that \(L'(f)(L(f_2)\xi) = L(f*f_2)\xi\) for all
  \(f,f_2\in\Contc(G^1)\), \(\xi\in\Hilm[F]_0\).  Since~\(L'\) is a
  \Star{}representation of the \Star{}algebra~\(\Contc(G^1)\)
  on~\(\Hilm[F]\),
  \[
  \braket{L'(f)\iota(\xi_1)}{L(f_2)\xi_2}
  = \braket{\iota(\xi_1)}{L'(f^*)L(f_2)\xi_2}
  = \braket{\iota(\xi_1)}{L(f^* * f_2)\xi_2}
  = \braket{L(f)\xi_1}{L(f_2)\xi_2}
  \]
  for all \(f,f_2\in\Contc(G^1)\), \(\xi_1,\xi_2\in\Hilm[F]_0\).
  Since vectors of the form \(L(f_2)\xi_2\) span a dense subspace
  in~\(\Hilm[F]\), this implies \(L'(f)\iota(\xi_1) = L(f)\xi_1\).
\end{proof}

\begin{corollary}
  \label{cor:prerep}
  For any pre-representation~\((L,\Hilm[F]_0,\iota)\)
  of~\(\Contc(G^1)\) there is a representation \(L'\colon
  \Contc(G^1)\to\Bound(\Hilm[F])\) that is bounded with respect to the
  \(I\)\nb-norm and satisfies \(L'(f)(\iota(\xi)) = L(f)(\xi)\) for
  all \(f\in\Contc(G^1)\), \(\xi\in\Hilm[F]_0\).  In particular, a
  representation of \(\Contc(G^1)\) is \(I\)\nb-norm bounded if and
  only if it is continuous in the inductive limit topology.
\end{corollary}

The first statement in Corollary~\ref{cor:prerep} looks
rather technical.  Nevertheless, it is a key result for the theory
of groupoid \(\Cst\)\nb-algebras.  For instance, it is needed to
prove that Morita equivalent groupoids have Morita--Rieffel
equivalent \(\Cst\)\nb-algebras, see
\cite{Muhly-Renault-Williams:Equivalence}*{p.~15}.  And this is
the only point in the proof of
\cite{Muhly-Renault-Williams:Equivalence}*{Theorem 2.8} where it
is needed that the groupoids involved are second countable.  So
our Corollary~\ref{cor:prerep} removes this hypothesis from the
main result of~\cite{Muhly-Renault-Williams:Equivalence}:

\begin{corollary}
  Let \((G,\alpha)\) and \((H,\beta)\) be locally compact
  groupoids with Haar systems.  Let~\(Z\) be a Morita equivalence
  between \(G\) and~\(H\).  Then \(\Contc(Z)\) may be completed to
  a \(\Cst(G,\alpha),\Cst(H,\beta)\)-imprimitivity bimodule.  So
  \(\Cst(G,\alpha)\) and \(\Cst(H,\beta)\) are Morita--Rieffel
  equivalent.
\end{corollary}

Now we finish the proof of our main theorem.

\begin{proof}[Proof of
  Theorem~\textup{\ref{the:universal_groupoid_Haus}}]
  A representation of~\(\Cst(G,\alpha)\) on~\(\Hilm[F]\) is
  equivalent to a nondegenerate \Star{}representation of the
  \Star{}algebra~\(\Contc(G^1)\) that is bounded with respect to the
  \(I\)\nb-norm; by Corollary~\ref{cor:prerep}, this is equivalent
  to continuity in the inductive limit topology.  When we
  disintegrate such a representation~\(L\) to a
  representation~\((\varphi,U)\) of~\((G,\alpha)\) and
  integrate~\((\varphi,U)\) to a representation of~\(\Contc(G^1)\),
  we get back the original representation~\(L\) by
  Proposition~\ref{pro:back_and_forth_from_L}.

  Now we start with a representation~\((\varphi,U)\)
  of~\((G,\alpha)\) and integrate it to a representation~\(L\)
  of~\(\Contc(G^1)\).  Let \((\varphi',U')\) be the representation
  of~\((G,\alpha)\) obtained by disintegrating~\(L\).  We claim that
  \(\varphi'=\varphi\) and \(U'=U\).  Let \(f_1,f_2\in\Contc(G^1)\),
  \(f_0\in\Contc(G^0)\) and \(\xi\in\Hilm[F]\).  Then \(f_1
  \rg^*(f_0) \otimes \xi = f_1 \otimes \varphi(f_0)\xi\) in
  \(\Lt^2(G^1,\rg,\alpha)\otimes_\varphi \Hilm[F]\).  Hence
  \(T_{\rg^*(f_0)\cdot f_1}^* = \varphi(f_0^*) T_{f_1}^*\) in
  \(\Bound(\Lt^2(G^1,\rg,\alpha)\otimes_\varphi
  \Hilm[F],\Hilm[F])\).  So
  \[
  \varphi(f_0) L(\conj{f_1}f_2)\xi
  = \varphi(f_0) T_{f_1}^* U T_{f_2}\xi
  = T_{\rg(f_0^*)f_1}^* U T_{f_2}\xi
  = L(\rg^*(f_0)\conj{f_1}f_2)\xi.
  \]
  Thus \(\varphi'=\varphi\) by the definition of~\(\varphi'\) in
  Lemma~\ref{lem:disintegrate_varphi}.

  Equation~\eqref{eq:integrated_form_explicit} implies
  \[
  \braket{f_1\otimes \xi_1}{U(f_2\otimes\xi_2)}
  = \braket{\xi_1}{T_{f_1}^* U T_{f_2}\xi_2}
  = \braket{\xi_1}{L(\conj{f_1} f_2)\xi_2}
  \]
  for all \(f_1,f_2\in\Contc(G^1)\) and \(\xi_1,\xi_2\in\Hilm[F]\).
  This shows that the representation~\(L\) uniquely determines~\(U\)
  because the inner products \(\braket{f_1\otimes
    \xi_1}{U(f_2\otimes\xi_2)}\) determine~\(U\).  Both
  \((\varphi,U)\) and~\((\varphi',U')\) integrate to the same
  representation~\(L\) by
  Proposition~\ref{pro:back_and_forth_from_L}.  This implies
  \(U=U'\) because~\(L\) determines~\(U\) uniquely.

  Let \(\Hilm[F]_1\) and~\(\Hilm[F]_2\) be Hilbert \(D\)\nb-modules
  with representations \((\varphi_1,U_1)\) and \((\varphi_2,U_2)\)
  of~\((G,\alpha)\) and let \(L_1\) and~\(L_2\) be the corresponding
  representations of \(\Contc(G^1)\).  Let \(J\colon
  \Hilm[F]_1\into\Hilm[F]_2\) be an isometry.  This intertwines
  \((\varphi_1,U_1)\) and \((\varphi_2,U_2)\) if and only if
  \(J\varphi_1(f) = \varphi_2(f)J\) for all \(f\in\Cont_0(G^0)\) and
  \((\Id_{\Lt^2(G^1,\rg,\alpha)} \otimes J) U_1 = U_2
  (\Id_{\Lt^2(G^1,\s,\tilde\alpha)} \otimes J)\).  Then
  \[
  L_2(f) J
  = T_h^* U_2 T_f J
  = T_h^* U_2 (\Id\otimes J) T_f
  = T_h^* (\Id\otimes J) U_1 T_f
  = J T_h^* U_1 T_f
  = J L_1(f),
  \]
  that is, \(J\) intertwines the representations \(L_1\) and~\(L_2\)
  of~\(\Contc(G^1)\).  Then it also intertwines the unique
  extensions of \(L_1\) and~\(L_2\) to \(\Cst(G,\alpha)\).

  Conversely, assume that~\(J\) intertwines two representations
  \(L_1\) and~\(L_2\) of \(\Contc(G^1)\) on \(\Hilm[F]_1\)
  and~\(\Hilm[F]_2\).  Then~\(J\) also intertwines the
  representations~\(\varphi_i\) of \(\Cont_0(G^0)\) defined in
  Lemma~\ref{lem:disintegrate_varphi} and the maps \(\Contc(G^2)
  \odot \Hilm[F]_i \to \Lt^2(G^1,\rg,\alpha) \otimes_{\varphi_i}
  \Hilm[F]_i\) and \(\Contc(G^2) \odot \Hilm[F]_i \to
  \Lt^2(G^1,\s,\tilde\alpha) \otimes_{\varphi_i} \Hilm[F]_i\) used
  to construct~\(U\); that is, the constructions in Lemmas
  \ref{lem:disintegrate_varphi} and Lemma~\ref{lem:dense_in_HilmF}
  are natural for isometric intertwiners in the formal sense.
  Thus~\(J\) intertwines the
  representations \((\varphi_1,U_1)\) and \((\varphi_2,U_2)\)
  obtained by disintegrating \(L_1\) and~\(L_2\).

  The previous two paragraphs show that the bijections between
  representations of \((G,\alpha)\) and~\(\Cst(G,\alpha)\) have the
  property~\ref{en:universal_groupoid_Haus1} in
  Theorem~\ref{the:universal_groupoid_Haus}.
  Property~\ref{en:universal_groupoid_Haus2} in
  Theorem~\ref{the:universal_groupoid_Haus} is obvious from our
  construction of the integrated form of a representation
  of~\((G,\alpha)\).
  The proof that our universal property
  characterises~\(\Cst(G,\alpha)\) uniquely up to a canonical
  isomorphism is the same as a corresponding argument for
  \(\Cst\)\nb-hulls, see \cite{Meyer:Unbounded}*{Proposition~3.7}.
\end{proof}

\begin{corollary}
  \label{cor:universal_rep}
  There is a \emph{universal
    representation}~\((\varphi^\univ,U^\univ)\) of~\((G,\alpha)\) on
  \(\Cst(G,\alpha)\) such that disintegration of representations
  maps a representation~\(L\) of \(\Cst(G,\alpha)\) on~\(\Hilm[F]\)
  to \((\varphi^\univ,U^\univ) \otimes_L \Hilm[F]\).  More
  precisely, this representation of~\((G,\alpha)\) lives
  on~\(\Cst(G,\alpha) \otimes_L \Hilm[F]\), which we identify
  with~\(\Hilm[F]\) by the canonical unitary \(f\otimes \xi\mapsto
  L(f)\xi\).
\end{corollary}

The proof also describes the universal
representation~\((\varphi^\univ,U^\univ)\).

\begin{proof}
  We view the identity map on~\(\Cst(G,\alpha)\) as a representation
  of~\(\Cst(G,\alpha)\) on~\(\Cst(G,\alpha)\) as a Hilbert
  module over itself.  This representation disintegrates to a
  representation~\((\varphi^\univ,U^\univ)\) of~\((G,\alpha)\)
  on~\(\Cst(G,\alpha)\).  By construction, \(\varphi^\univ\colon
  \Cont_0(G^0)\to \Bound(\Cst(G,\alpha)) = \Mult(\Cst(G,\alpha))\)
  is the canonical morphism: a function~\(f_0\) in \(\Cont_0(G^0)\)
  multiplies on the left by~\(\rg^*(f_0)\) and hence on the right
  by~\(\s^*(f_0)\).  The unitary
  \[
  U^\univ\colon
  \Lt^2(G,\s,\tilde\alpha) \otimes_{\varphi^\univ} \Cst(G,\alpha)
  \congto
  \Lt^2(G,\rg,\alpha) \otimes_{\varphi^\univ} \Cst(G,\alpha)
  \]
  is the unique extension of the isomorphism
  \begin{equation}
    \label{eq:U_univ_dense}
    \Contc(G^1\times_{\s,G^0,\rg} G^1) \congto
    \Contc(G^1\times_{\rg,G^0,\rg} G^1)
  \end{equation}
  that composes functions with the canonical homeomorphism
  \[
  G^1\times_{\rg,G^0,\rg} G^1 \congto G^1\times_{\s,G^0,\rg} G^1,
  \qquad
  (g,k)\mapsto (g,g^{-1} k).
  \]
  A variant of Lemma~\ref{lem:dense_in_HilmF} gives linear maps
  \begin{align*}
    \Contc(G^1\times_{\s,G^0,\rg} G^1) &\to
    \Lt^2(G,\s,\tilde\alpha) \otimes_{\varphi^\univ} \Cst(G,\alpha),\\
    \Contc(G^1\times_{\rg,G^0,\rg} G^1) &\to
    \Lt^2(G,\rg,\alpha) \otimes_{\varphi^\univ} \Cst(G,\alpha)
  \end{align*}
  with dense image and shows that the
  isomorphism~\eqref{eq:U_univ_dense} preserves the inner products.
  More precisely, Lemma~\ref{lem:dense_in_HilmF} considers
  \(\Contc(G^2)\odot \Cst(G,\alpha)\).  But we may do the same
  computation without the factor~\(\Cst(G,\alpha)\), getting the
  above, simpler, description of~\(U^\univ\).

  The claim that disintegration is just tensoring a given
  representation of~\(\Cst(G,\alpha)\)
  with the universal representation is implicit in our construction in
  Section~\ref{sec:disintegrate}.  We deduce it from the two extra
  properties of the bijections in
  Theorem~\ref{the:universal_groupoid_Haus}, compare the proof of
  \cite{Meyer:Unbounded}*{Proposition~3.6}.  The canonical unitary
  \(\Cst(G,\alpha) \otimes_L \Hilm[F]\congto \Hilm[F]\),
  \(f\otimes \xi\mapsto L(f)\xi\),
  intertwines the obvious representations of~\(\Cst(G,\alpha)\)
  on both Hilbert modules.  It intertwines the corresponding
  representations of~\((G,\alpha)\)
  by
  \ref{the:universal_groupoid_Haus}.\ref{en:universal_groupoid_Haus1}.
  The representation of~\((G,\alpha)\)
  on~\(\Cst(G,\alpha) \otimes_L \Hilm[F]\)
  is \((\varphi^\univ,U^\univ) \otimes_L \Hilm[F]\)
  by
  \ref{the:universal_groupoid_Haus}.\ref{en:universal_groupoid_Haus2}.
  Hence the disintegration of~\(L\)
  is obtained by transporting the representation
  \((\varphi^\univ,U^\univ) \otimes_L \Hilm[F]\)
  on \(\Cst(G,\alpha) \otimes_L \Hilm[F]\)
  to one on~\(\Hilm[F]\)
  along the canonical unitary
  \(\Cst(G,\alpha) \otimes_L \Hilm[F]\congto \Hilm[F]\).
\end{proof}

\section{Transformation groups and \'etale groupoids}
\label{sec:trafo_group_etale}

We now make our universal property more explicit in two cases,
namely, for transformation groups and for (Hausdorff) \'etale
groupoids.  In both cases, there is a canonical Haar
system~\(\alpha\).  We are going to reprove known
characterisations of \(\Cst(G,\alpha)\) as a crossed product for a
group or inverse semigroup action, respectively.

\subsection{Transformation groups}
Let~\(\Gamma\)
be a locally compact group and let \(X\)
be a left \(\Gamma\)\nb-space.
Let \(G=\Gamma\ltimes X\)
be the transformation groupoid.  Fix a Haar measure~\(\alpha_0\)
on~\(\Gamma\)
and let~\(\alpha\)
be the resulting ``constant'' Haar system on~\(G\).  By definition,
\begin{alignat*}{3}
  G^0 &= X,&\qquad
  \s(\gamma,x) &= x,&\qquad
  d_0(\gamma_1,\gamma_2,x) &= (\gamma_2,x),\\
  G^1 &= \Gamma\times X,&\qquad
  \rg(\gamma,x) &= \gamma\cdot x,&\qquad
  d_1(\gamma_1,\gamma_2,x) &= (\gamma_1 \gamma_2, x),\\
  G^2 &= \Gamma^2\times X,&\qquad
  && d_2(\gamma_1,\gamma_2,x) &= (\gamma_1, \gamma_2 x).
\end{alignat*}
The measure family~\(\tilde\alpha\)
is the constant family
\(\tilde\alpha_x = \tilde\alpha_0\times\delta_x\)
for \(x\in X\),
where~\(\tilde\alpha_0\)
is the right Haar measure on~\(\Gamma\)
associated to~\(\alpha_0\).

\begin{theorem}
  \label{the:rep_trafo_group}
  Let~\((G,\alpha)\) be a transformation groupoid with its
  canonical Haar system.  Let~\(D\) be a \(\Cst\)\nb-algebra
  and~\(\Hilm[F]\) a Hilbert \(D\)\nb-module.  There is a
  bijection between representations of \((G,\alpha)\)
  on~\(\Hilm[F]\) and covariant representations for the action
  of~\(\Gamma\) on~\(\Cont_0(X)\).
  It leaves the representation of~\(\Cont_0(X)\) unchanged and has
  the two naturality properties in
  Theorem~\textup{\ref{the:universal_groupoid_Haus}}.
  Thus \(\Cst(G,\alpha)\) is naturally isomorphic to
  the crossed product \(\Gamma \ltimes \Cont_0(X)\).
\end{theorem}

\begin{proof}
  A representation of~\((G,\alpha)\)
  consists of a representation~\(\varphi\)
  of \(\Cont_0(G^0)=\Cont_0(X)\) on~\(\Hilm[F]\) and a unitary operator
  \[
  U\colon \Lt^2(G^1,\s,\tilde\alpha) \otimes_{\Cont_0(X)} \Hilm[F]
  \congto \Lt^2(G^1,\rg,\alpha) \otimes_{\Cont_0(X)} \Hilm[F]
  \]
  with some properties.  We fix~\(\varphi\) because it appears in
  both types of representations that we are going to identify.
  Since \(G^1 = \Gamma\times X\),
  the Hilbert \(\Cont_0(X)\)-module
  \(\Lt^2(G^1,\s,\tilde\alpha)\)
  is isomorphic to the exterior tensor product
  \(\Lt^2(\Gamma,\tilde\alpha_0)\otimes \Cont_0(X)\);
  this corresponds to the constant field of Hilbert spaces over~\(X\)
  with fibre~\(\Lt^2(\Gamma,\tilde\alpha_0)\).  Hence we may simplify
  \[
  \Lt^2(G^1,\s,\tilde\alpha) \otimes_{\Cont_0(X)} \Hilm[F] \cong
  \Lt^2(\Gamma,\tilde\alpha_0) \otimes \Hilm[F].
  \]
  This exterior tensor product contains \(\Contc(\Gamma,\Hilm[F])\)
  as a dense subspace.  On this subspace, the representation of
  \(\Cont_0(G^1) = \Cont_0(\Gamma\times X)\) acts by
  \[
  f\cdot \xi(\gamma) = \varphi\bigl(x\mapsto f(\gamma,x)\bigr)
  \xi(\gamma)
  \]
  for all \(f\in\Cont_0(\Gamma\times X)\),
  \(\xi\in\Contc(\Gamma,\Hilm[F])\);
  that is, for fixed \(\gamma\in \Gamma\),
  \(f\)
  acts by~\(\varphi\)
  applied to the restricted function \(x\mapsto f(\gamma,x)\).

  Next we need a similar unitary operator
  \[
  \Lt^2(G^1,\rg,\alpha) \otimes_{\Cont_0(X)} \Hilm[F] \congto
  \Lt^2(\Gamma,\alpha_0)\otimes \Hilm[F].
  \]
  We construct it in three steps.  First, the inversion in~\(G\)
  induces a unitary \(\Lt^2(G^1,\rg,\alpha) \cong
  \Lt^2(G^1,\s,\tilde\alpha)\), which we may tensor
  with~\(\Hilm[F]\).  Then the unitary above maps this on to
  \(\Lt^2(\Gamma,\tilde\alpha_0) \otimes \Hilm[F]\).  Finally, the
  inversion in~\(\Gamma\) induces a unitary from
  \(\Lt^2(\Gamma,\tilde\alpha_0)\) to \(\Lt^2(\Gamma,\alpha_0)\).
  Since \((\gamma,x)^{-1} = (\gamma^{-1},\gamma x)\), this chain
  of unitaries may be described by the single homeomorphism
  \[
  \iota\colon G^1 \to \Gamma \times X,\qquad
  (\gamma,x)\mapsto (\gamma,\gamma x).
  \]
  The decomposition above shows that this homeomorphism induces a
  unitary operator
  \[
  \iota_*\colon
  \Lt^2(G^1,\rg,\alpha) \congto
  \Lt^2(\Gamma,\alpha_0)\otimes \Cont_0(X),
  \qquad f\mapsto f\circ\iota^{-1}.
  \]
  It intertwines the left actions of~\(\Cont_0(G^1)\) on
  \(\Lt^2(G^1,\rg,\alpha)\) by pointwise multiplication and on
  \(\Lt^2(\Gamma,\alpha)\otimes \Cont_0(X)\) by \((f \bullet
  \xi)(\gamma,x) = f(\gamma,\gamma x)\cdot \xi(\gamma,x)\).  The
  unitary~\(\iota_*\) induces a unitary operator
  \[
  \Lt^2(G^1,\rg,\alpha) \otimes_{\Cont_0(X)} \Hilm[F] \congto
  \Lt^2(\Gamma,\alpha_0)\otimes \Hilm[F],
  \]
  which intertwines the representations of~\(\Cont_0(G^1)\) on
  \(\Lt^2(G^1,\rg,\alpha) \otimes_{\Cont_0(X)} \Hilm[F]\) by
  pointwise multiplication and on \(\Lt^2(\Gamma,\alpha_0)\otimes
  \Hilm[F]\) by
  \[
  (f\bullet \xi)(\gamma)
  = \varphi\bigl(x\mapsto f(\gamma,\gamma x)\bigr) \xi(\gamma)
  \]
  for \(f\in\Cont_0(\Gamma\times X)\),
  \(\xi\in\Contc(\Gamma,\Hilm[F])\).

  Using the two unitaries above, we identify the unitary
  operator~\(U\)
  in a representation of~\((G,\alpha)\) with a unitary operator
  \[
  U'\colon \Lt^2(\Gamma,\tilde\alpha_0) \otimes \Hilm[F]
  \congto \Lt^2(\Gamma,\alpha_0) \otimes \Hilm[F],
  \]
  which intertwines the representations of~\(\Cont_0(\Gamma\times X)\)
  specified above and makes a certain diagram commute.  Since
  \(\Cont_0(\Gamma\times X) \cong \Cont_0(\Gamma)\otimes \Cont_0(X)\),
  the intertwining condition is equivalent to intertwining conditions
  for the representations of \(\Cont_0(\Gamma)\)
  and~\(\Cont_0(X)\)
  that we get by taking functions~\(f\)
  above that depend only on the first or second variable,
  respectively.

  We may identify
  \(\Lt^2(G^2,v_i,\mu_i)\otimes_{\Cont_0(G^0)} \Hilm[F] \cong
  \Lt^2(\Gamma^2,\mu'_i) \otimes \Hilm[F]\),
  where the measure~\(\mu_i'\)
  is constructed like~\(\mu_i\)
  but for the group~\(\Gamma\).
  The commuting diagram needed for~\(U\)
  to be a representation is equivalent to the corresponding commuting
  diagram for~\(U'\).
  Thus~\(U'\)
  is a representation of~\((\Gamma,\alpha_0)\)
  if~\(U\)
  is a representation of~\((G,\alpha)\).
  Then~\(U'\)
  comes from a continuous representation of~\(\Gamma\)
  on~\(\Hilm[F]\)
  because~\(\Cst(\Gamma)\)
  is also universal with respect to continuous representations
  of~\(\Gamma\).
  The condition that~\(U'\)
  intertwines the two representations of~\(\Cont_0(X)\)
  means that the representation~\(\varphi\)
  of~\(\Cont_0(X)\)
  is covariant with respect to the continuous representation
  of~\(\Gamma\)
  associated to~\(U'\).

  The above construction may be reversed easily.  So there is a
  bijection between representations of~\((G,\alpha)\) and covariant
  representations for the action of~\(\Gamma\) on~\(\Cont_0(X)\).
  There is a bijection between the latter and
  representations of the crossed product \(\Gamma \ltimes
  \Cont_0(X)\), see \cite{Williams:crossed-products}*{Proposition
    2.39}.  Using also Theorem~\ref{the:universal_groupoid_Haus},
  we get a bijection between representations of \(\Cst(G,\alpha)\)
  and \(\Gamma \ltimes \Cont_0(X)\) on~\(\Hilm[F]\).  All
  bijections used above have the naturality properties in
  Theorem~\ref{the:universal_groupoid_Haus}.  As in the proof of
  \cite{Meyer:Unbounded}*{Proposition~3.7}, this implies that
  there is a natural isomorphism \(\Cst(G,\alpha)\cong\Gamma
  \ltimes \Cont_0(X)\) that induces the bijection on
  representations built above.
\end{proof}

\subsection{Étale groupoids}
\label{sec:etale}

Let~\(G\) be a locally compact, Hausdorff and étale groupoid.
Endow~\(G\) with the canonical Haar system~\(\alpha\) where
each~\(\alpha^x\) is the counting measure on the discrete subset
\(G^x\subseteq G^1\).  Then~\(\tilde\alpha\) is the family of
counting measures~\(\tilde\alpha_x\) on the discrete subsets
\(G_x\subseteq G^1\).  Similarly, all the other measures
\(\lambda_i\) and~\(\mu_i\) associated to the maps \(d_i\)
and~\(v_i\) for \(i=0,1,2\) are counting measures.

We quickly recall the relationship between étale groupoids and
inverse semigroup actions on spaces, see
\cites{Exel:Inverse_combinatorial,Paterson:Groupoids}.  A
\emph{bisection} of~\(G\) is an open subset \(a\subseteq G^1\) such
that the restrictions of the source and range maps
\[
s_a\colon a\to \s(a),\qquad \rg_a\colon a\to \rg(a),
\]
are injective or, equivalently, homeomorphisms onto their images.
The bisections in~\(G\) form an inverse semigroup \(\Bis(G)\).  It
acts on~\(G^0\) by the partial homeomorphisms \(\vartheta_a =
\rg_a\circ\s_a^{-1}\colon \s(a)\congto \rg(a)\).  The transformation
groupoid \(\Bis(G)\ltimes G^0\) for this action is naturally
isomorphic to~\(G\) (see also~\cite{Exel:Partial_dynamical}).  Here
we may replace \(\Bis(G)\) by any inverse subsemigroup \(S\subseteq
\Bis(G)\) that is \emph{wide}, that is,
\begin{equation}
  \label{eq:wide_isg}
  \bigcup_{t\in S} t = G^1,\qquad
  \bigcup_{\substack{v\in S\\v \le u,t}} v = u\cap t\qquad
  \text{for all }u,t\in S.
\end{equation}
These conditions say that the groupoid
homomorphism \(S\ltimes G^0 \to \Bis(G)\ltimes G^0 \cong G^1\)
induced by the inclusion \(S\to \Bis(G)\) is bijective.  Then it is
a homeomorphism because it is always a local homeomorphism.  That
is, \(G\cong S\ltimes G^0\) as topological groupoids if and only
if~\(S\) is
wide.  We fix an inverse semigroup action~\(\vartheta\) of an inverse
semigroup~\(S\) on a locally compact space \(X=G^0\) and an
isomorphism \(G\cong S\ltimes_\vartheta X\).  The action of~\(S\)
on~\(X\) induces an action on the \cstar{}algebra \(\Cont_0(X)\) by
partial isomorphisms.  This action has a crossed
product \(S\ltimes\Cont_0(X)\).  It is already known that \(\Cst(G)
\cong S\ltimes\Cont_0(X)\), see
\cite{Exel:Inverse_combinatorial}*{Theorem~9.8}, for instance.  The
same result is also proved in \cites{Paterson:Groupoids,
  Sieben-Quigg:ActionsOfGroupoidsAndISGs,
  Buss-Meyer:Actions_groupoids,
  BussExel:Fell.Bundle.and.Twisted.Groupoids}, sometimes under mild
extra conditions.  Our universal property gives another proof,
assuming~\(G\) to be Hausdorff.  The main point is that
representations of~\(G\) are equivalent to covariant
representations of the \(S\)\nb-action on \(\Cont_0(X)\), defined as
follows:

\begin{definition}
  \label{def:cov-rep}
  Let~\(X\) be a locally compact Hausdorff space with an
  action of an inverse semigroup~\(S\) by partial homeomorphisms
  \(\vartheta_a\colon D_{a^*a}\congto D_{aa^*}\) between open subsets
  \(D_e\subseteq X\) for \(e\in E(S)\).  A \emph{covariant
    representation} of this system on a Hilbert
  \(D\)\nb-module~\(\Hilm[F]\) consists of a nondegenerate representation
  \(\varphi\colon \Cont_0(X)\to \Bound(\Hilm[F])\) and a family of
  unitaries \(U_a\colon \Hilm[F]_{a^*a}\congto \Hilm[F]_{aa^*}\) for \(a\in S\),
  where \(\Hilm[F]_e\defeq \varphi(\Cont_0(D_e))\Hilm[F]\), that satisfy the
  following conditions:
  \begin{enumerate}
  \item \label{en:cov-rep1} \(U_b\) restricts to~\(U_a\) if \(a\le
    b\) in~\(S\);
  \item \label{en:cov-rep2} \(U_a^* = U_{a^*}\) for all \(a\in S\);
  \item \label{en:cov-rep3} \(U_a U_b = U_{ab}\) for all \(a,b\in S\)
    with \(a^*a=bb^*\);
  \item \label{en:cov-rep4} \(U_a\varphi(f)U_a^* =
    \varphi(f\circ\vartheta_{a^*})\) for all \(f\in
    \Cont_0(D_{a^*a})\).
  \end{enumerate}
\end{definition}

If~\(S\) is a group, so that \(1\in S\) is the only idempotent
element, then covariant representations
as defined above are equivalent to representations of the crossed
product \(S\ltimes\Cont_0(X)\).  For inverse semigroups, it seems
that covariant representations have so far only been introduced on
Hilbert \emph{spaces}; the standard reference for this
is~\cite{Sieben:crossed.products}.  Hilbert module representations
require extra care because the Hilbert submodules~\(\Hilm[F]_e\) need not
be complementable.  For instance, let~\(S\) be a semilattice of open
subsets~\(D_e\) of~\(X\) with the intersection as product, acting
on~\(X\) by identity maps \(\vartheta_e\colon D_e\congto D_e\).  The
identity map \(\varphi\colon \Cont_0(X)\congto\Cont_0(X)\) and the
identity maps \(U_e\colon \Cont_0(D_e)\congto
\Cont_0(D_e)\) form a covariant representation of this action on
\(\Hilm[F]=\Cont_0(X)\) viewed as a Hilbert \(\Cont_0(X)\)-module.  The
submodule \(\Hilm[F]_e = \Cont_0(D_e)\) is only complementable in~\(\Hilm[F]\)
if~\(D_e\) is clopen.  Hence we cannot replace the partially defined
maps~\(U_a\) above by partial isometries on~\(\Hilm[F]\).  The following
proposition asserts that our notion of covariant representation,
when specialised to Hilbert spaces, is equivalent to Sieben's
\cite{Sieben:crossed.products}*{Definition~3.4}.

\begin{proposition}
  \label{pro:isg_covariant_rep}
  Let \((X,S,\vartheta)\) be an action of~\(S\) on~\(X\) and
  let~\(\Hils\) be a Hilbert space.  Then a covariant representation
  of~\((X,S,\vartheta)\) on~\(\Hils\) is equivalent to a nondegenerate
  representation \(\varphi\colon \Cont_0(X)\to \Bound(\Hils)\)
  together with a
  map \(S\ni a\mapsto U_a\in \Bound(\Hils)\) with
  \(U_e(\Hils)=\varphi(D_e)\Hils\) for all \(e\in E(S)\) and
  \(U_a^*=U_{a^*}\), \(U_{ab} = U_a U_b\) and
  \(\varphi(f\circ\vartheta_{a^*}) = U_a\varphi(f)U_a^*\) for all \(a,b\in
  S\) and \(f\in \Cont_0(D_{a^*a})\subseteq \Cont_0(X)\).
\end{proposition}

\begin{proof}
  Let \(\varphi\) and \((U_a)_{a\in S}\) be as in the statement.
  Each~\(U_a\) is a partial isometry of~\(\Hils\) with source
  projection \(U_a^* U_a = U_{a^*a}\), the orthogonal projection
  onto \(\Hils_{a^*a}\defeq \varphi(\Cont_0(D_{a^*a}))\Hils\), and range
  projection \(U_a U_a^*\), the orthogonal projection
  onto~\(\Hils_{aa^*}\).  Hence we may
  view~\(U_a\) as an isomorphism \(\Hils_{a^*a}\congto
  \Hils_{aa^*}\).  In this way, \(\varphi\) and the isomorphisms
  \(U_a\colon \Hils_{a^*a}\congto \Hils_{aa^*}\) form a covariant
  representation as in Definition~\ref{def:cov-rep}.
  Conversely, let \(\varphi\colon \Cont_0(X)\to \Bound(\Hils)\) and
  \(U_a\colon \Hils_{a^*a}\congto \Hils_{aa^*}\) for \(a\in S\) form
  a covariant representation of \((X,S,\vartheta)\) as in
  Definition~\ref{def:cov-rep}, where \(\Hils_e \defeq
  \varphi(\Cont_0(D_e))\Hils\).  Extend~\(U_a\) by zero on the
  orthogonal complement of~\(\Hils_{a^*a}\) to view it as a partial
  isometry \(U_a\in \Bound(\Hils)\).  It remains to show that \(U_a
  U_b = U_{a b}\) for all \(a,b\in S\).  We already have this
  relation if \(a^* a = b b^*\).  And we have assumed that~\(U_b\)
  restricts to~\(U_a\) if \(a\le b\).

  Let \(e,f\in E(S)\).  Then \(\Hils_{ef}\subseteq \Hils_e\cap
  \Hils_f\).  Conversely, if \(\xi \in \Hils_e\cap
  \Hils_f\), then \(\varphi(u_i^e)\xi\to \xi\) for an approximate
  unit~\((u_i^e)\) of \(\Cont_0(D_e)\) and \(\varphi(u_i^f)\xi\to
  \xi\) for an approximate unit~\((u_i^f)\) of \(\Cont_0(D_f)\).
  Since \(u_i^e \cdot u_i^f \in \Cont_0(D_{e f})\), this implies
  \(\xi \in \Hils_{e f}\).  Thus \(\Hils_e\cap \Hils_f = \Hils_{e
    f}\).  This is equivalent to \(U_e U_f = U_{e f}\).  Using this,
  the assumption that~\(U_b\) restricts to~\(U_a\) if \(a\le b\)
  implies \(U_a U_e = U_{a e}\) for all \(a\in S\), \(e\in E(S)\).
  Thus
  \[
  U_a U_b
  = U_a U_{a^* a} U_{b b^*} U_b
  = U_a U_{b b^*} U_{a^* a} U_b
  = U_{a b b^*} U_{a^* a b}
  = U_{a b b^* a^* a b}
  = U_{a b}
  \]
  for all \(a,b\in S\) because \(a (a^* a) = a\), \(b (b^* b) = b\)
  and \((a b b^*) (a^* a b) = a b\) with matching range and source
  projections.
\end{proof}

\begin{theorem}
  \label{the:isg_crossed_rep_Hilm}
  There are bijections between covariant representations of
  \((X,S,\vartheta)\)
  and representations of
  \(S\ltimes \Cont_0(X)\)
  on Hilbert modules~\(\Hilm[F]\)
  over arbitrary \cstar{}algebras~\(D\).
  These have the naturality properties in
  Theorem~\textup{\ref{the:universal_groupoid_Haus}}.
\end{theorem}

\begin{proof}
  For Hilbert space representations, this follows from
  Proposition~\ref{pro:isg_covariant_rep} and the definition of the
  crossed product, compare~\cite{Sieben:crossed.products}.  We prove
  it for general~\(\Hilm[F]\).  It is no loss of generality to
  assume that~\(S\) has a unit element \(1\in S\) because we may
  always add a formal unit to~\(S\) and extend the action to the
  unitisation without changing the crossed product.  The crossed
  product \(S\ltimes\Cont_0(X)\) is the universal \cstar{}algebra
  generated by expressions of the form~\(f_a\delta_a\) with \(a\in
  S\) and \(f_a\in \Cont_0(D_{aa^*})\), subject to the relations
  that \(f_a\mapsto f_a\delta_a\) is a linear map
  \(\Cont_0(D_{aa^*})\to S\ltimes\Cont_0(X)\), that
  \(f\delta_a=f\delta_b\) if \(a\leq b\) in \(S\) and \(f\in
  \Cont_0(D_{aa^*})\), plus the following algebraic relations:
  \[
  (f_a\delta_a)\cdot (f_b\delta_b)
  =\bigl(((f_a\circ\vartheta_a)\cdot f_b)\circ
  \vartheta_{a^*}\bigr)\delta_{ab},\qquad
  (f_a\delta_a)^*
  = (\conj f_a\circ \vartheta_{a})\delta_{a^*}
  \]
  for all \(a,b\in S\), \(f_a\in \Cont_0(D_{aa^*})\) and \(f_b\in
  D_{bb^*}\).  A covariant representation \((\varphi,(U_a)_{a\in
    S})\) of \((X,S,\vartheta)\) on a Hilbert module~\(\Hilm[F]\) integrates
  to a representation \(U\ltimes\varphi\colon S\ltimes\Cont_0(X)\to
  \Bound(\Hilm[F])\) by \(U\ltimes\varphi(f_a\delta_a)\defeq
  \varphi(f_a)U_a\).  The covariance conditions imply that this is a
  well defined, nondegenerate \Star{}homomorphism.  Conversely, given
  a representation \(\varrho\colon S\ltimes\Cont_0(X)\to \Bound(\Hilm[F])\),
  we define \(\varphi\colon \Cont_0(X)\to \Bound(\Hilm[F])\) by
  \(\varphi(f)\defeq \varrho(f\delta_1)\).  Then~\(\varphi\) is a nondegenerate
  \Star{}homomorphism.  Given \(a\in S\), we define \(U_a\colon
  \Hilm[F]_{a^*a}\to \Hilm[F]_{aa^*}\) by \(U_a(\varrho(f\delta_{a^*a})\xi)\defeq
  \varrho((f\circ\vartheta_{a^*})\delta_a)\xi\) for all \(f\in
  \Cont_0(D_{a^*a})\) and \(\xi\in \Hilm[F]\).  By definition,
  \(\Hilm[F]_{a^*a}\) consists of elements of the form
  \(\varrho(f\delta_{a^*a})\xi=\varphi(f)\xi\) with \(f\in
  \Cont_0(D_{a^*a})\) and \(\xi\in \Hilm[F]\).  Writing~\(f\) as a product
  of two elements of \(\Cont_0(D_{a^*a})\) and using the definition
  of~\(\varphi\), we get \(\varrho((f\circ\vartheta_{a^*})\delta_a)\xi\in
  \Hilm[F]_{aa^*}\).  The relation
  \[
  ((f_1\circ\vartheta_{a^*})\delta_a)^*\cdot (f_2\circ\vartheta_{a^*})\delta_a
  = (\conj f_1\cdot f_2)\delta_{a^*a}
  \]
  holds in \(S\ltimes\Cont_0(X)\) for all \(f_1,f_2\in
  \Cont_0(D_{a^*a})\).  Hence the map~\(U_a\) is a well defined
  isometry \(\Hilm[F]_{a^*a}\to\Hilm[F]_{aa^*}\):
  \begin{multline*}
    \braket{\varrho((f_1\circ\vartheta_{a^*})\delta_a)\xi_1}
    {\varrho((f_2\circ\vartheta_{a^*})\delta_a)\xi_2}
    = \braket{\xi_1}{\varrho(((f_1\circ\vartheta_{a^*})\delta_a)^*\cdot
      (f_2\circ\vartheta_{a^*})\delta_a)\xi_2}\\
    = \braket{\xi_1}{\varrho((\conj f_1\cdot f_2)\delta_{a^*a})\xi_2}
    = \braket{\varrho(f_1\delta_{a^*a})\xi_1}{\varrho(f_2\delta_{a^*a})\xi_2}.
  \end{multline*}
  By definition, the image of~\(U_a\) is
  \[
  \varrho(\Cont_0(D_{aa^*})\delta_a)\Hilm[F] \subseteq \Hilm[F]_{aa^*}
  = \varphi(\Cont_0(D_{aa^*}))\Hilm[F]
  = \varrho(\Cont_0(D_{aa^*})\delta_{aa^*})\Hilm[F].
  \]
  Indeed, this inclusion
  is an equality so that~\(U_a\) is unitary; the other inclusion
  follows because \(\Cont_0(D_{aa^*})\delta_{aa^*} =
  (\Cont_0(D_{aa^*})\delta_a)\cdot (\Cont_0(D_{a^*a})\delta_{a^*})\)
  in \(S\ltimes\Cont_0(X)\).  By construction, \(U_a\) satisfies
  \(U_a\varphi(f)U_{a^*} = \varphi(f\circ\vartheta_{a^*})\) for all
  \(f\in \Cont_0(D_{a^*a})\).  Clearly, \(U_a^*=U_{a^*}\) and
  \(U_a\) restricts to~\(U_b\) if \(b\leq a\).  The remaining
  multiplicativity property \(U_a U_b = U_{a b}\) for \(a^* a = b
  b^*\) is also easily checked.  Therefore, \(\varphi\) and the
  partial unitaries~\(U_a\) for \(a\in S\) form a covariant
  representation of \((X,S,\vartheta)\).  By construction,
  \(U\ltimes\varphi=\varrho\).  Thus \((\varphi,U)\mapsto
  U\ltimes\varphi\) implements the desired bijection between
  covariant representations and representations of the crossed
  product \(S\ltimes\Cont_0(X)\).
  We leave it to the reader to check that these bijections have the
  naturality properties in Theorem~\ref{the:universal_groupoid_Haus}.
\end{proof}

\begin{theorem}
  \label{the:etale_covariant_rep}
  Let~\(S\) be an inverse semigroup and let~\(X\) be a locally
  compact space with an action~\(\vartheta\) of~\(S\) by partial
  homeomorphisms.  Let \(G\defeq S\ltimes X\).  Assume that this
  étale, locally compact groupoid is Hausdorff.  Let~\(\Hilm[F]\) be a
  Hilbert \(D\)\nb-module.  Representations of~\(G\) on~\(\Hilm[F]\)
  correspond bijectively to covariant representations of
  \((X,S,\vartheta)\) on~\(\Hilm[F]\), which correspond bijectively to
  representations of \(S \ltimes_\vartheta \Cont_0(X)\) on~\(\Hilm[F]\).
  These bijections have the naturality properties in
  Theorem~\textup{\ref{the:universal_groupoid_Haus}}.  So they come
  from a unique isomorphism \(\Cst(G) \cong S
  \ltimes_\vartheta \Cont_0(X)\).
\end{theorem}

\begin{proof}
  First we construct a covariant representation of \((X,S,\vartheta)\)
  on~\(\Hilm[F]\) from a representation \((U,\varphi)\) of \(S\ltimes X\)
  on~\(\Hilm[F]\).  Here \(\varphi\colon \Cont_0(G^0)\to \Bound(\Hilm[F])\) is a
  nondegenerate representation and~\(U\) is an isomorphism of
  \(\Cont_0(G^1)\)-\(D\)-correspondences
  \[
  U\colon \Lt^2(G^1,\s,\tilde\alpha)\otimes_\varphi\Hilm[F]\congto
  \Lt^2(G^1,\rg,\alpha)\otimes_\varphi\Hilm[F]
  \]
  with \(d_2^*(U) d_0^*(U) = d_1^*(U)\).  For an open subset
  \(e\subseteq G^0\), let \(\Hilm[F]_e\defeq \varphi(\Cont_0(e))\cdot
  \Hilm[F]\).  Any \(a\in S\) gives a bisection of~\(G^1\), namely, the
  set of all germs of pairs~\((a,x)\) with \(x\in D_{a^* a}\).  By
  abuse of notation, we also denote this bisection by~\(a\).  The
  isomorphism~\(U\) restricts to an isomorphism of
  \(\Cont_0(a)\)-\(D\)-correspondences
  \[
  U|_a\colon \Cont_0(a)\cdot
  \Lt^2(G^1,\s,\tilde\alpha)\otimes_\varphi\Hilm[F]\congto \Cont_0(a)\cdot
  \Lt^2(G^1,\rg,\alpha)\otimes_\varphi\Hilm[F].
  \]
  There are canonical isomorphisms of
  \(\Cont_0(a)\)-\(\Cont_0(G^0)\)-correspondences
  \[
  \Cont_0(a)\cdot \Lt^2(G^1,\s,\tilde\alpha)\cong \Cont_0(\s(a)),
  \qquad
  \Cont_0(a)\cdot \Lt^2(G^1,\rg,\alpha)\cong \Cont_0(\rg(a)).
  \]
  The first isomorphism sends a function \(\xi\in \Cont_0(a)\cdot
  \Contc(G^1)=\Contc(a)\) to \(\xi\circ \s_a^{-1}\in
  \Contc(\s(a))\), and similarly for the second.  Here we view the
  ideal \(\Cont_0(\s(a))\) in~\(\Cont_0(G^0)\) as a
  \(\Cont_0(a)\)-\(\Cont_0(G^0)\)-correspondence in the canonical
  way, using the homeomorphism~\(s_a\).  This gives canonical
  isomorphisms of \(\Cont_0(a)\)-\(D\)-correspondences
  \begin{align*}
    \Cont_0(a)\cdot \Lt^2(G^1,\s,\tilde\alpha)\otimes_\varphi\Hilm[F]
    &\cong \Cont_0(\s(a))\otimes_\varphi\Hilm[F]
    \cong \varphi(\Cont_0(\s(a)))\Hilm[F]
    = \Hilm[F]_{\s(a)},\\
    \Cont_0(a)\cdot \Lt^2(G^1,\rg,\alpha)\otimes_\varphi\Hilm[F]
    &\cong \Cont_0(\rg(a))\otimes_\varphi\Hilm[F]
    \cong \varphi(\Cont_0(\rg(a)))\Hilm[F]
    = \Hilm[F]_{\rg(a)}.
  \end{align*}
  Thus~\(U|_a\) becomes an isomorphism of
  \(\Cont_0(a)\)-\(D\)-correspondences
  \[
  U_a\colon \Hilm[F]_{\s(a)}\congto \Hilm[F]_{\rg(a)},
  \]
  where we view \(\Hilm[F]_{\s(a)}\) and~\(\Hilm[F]_{\rg(a)}\) as
  \(\Cont_0(a)\)-\(D\)-correspondences using the homeomorphisms
  \(\s_a\) and~\(\rg_a\).  That is,
  \[
  U_a(\varphi(f\circ s_a^{-1})\xi)
  = \varphi(f\circ \rg_a^{-1}) U_a(\xi)\qquad
  \text{for all }f\in \Cont_0(a),\ \xi\in \Hilm[F]_{\s(a)}.
  \]
  We reinterpret this as a covariance condition.  The inverse
  semigroup~\(S\) acts on the \cstar{}algebra \(\Cont_0(G^0)\) by
  the isomorphisms
  \[
  \gamma_a\colon \Cont_0(\s(a))\congto \Cont_0(\rg(a)),\qquad
  f\mapsto f\circ \vartheta_{a^*} = f\circ\s_a\circ\rg_a^{-1}.
  \]
  Thus \(U_a \varphi(f) U_a^* = \varphi(\gamma_a(f))\) for all
  \(f\in \Cont_0(\s(a))\) as in
  Definition~\ref{def:cov-rep}.\ref{en:cov-rep4}.  The conditions in
  \ref{en:cov-rep1} and~\ref{en:cov-rep2} in
  Definition~\ref{def:cov-rep} are also clear from our construction.
  We claim that the condition \(d_2^*(U)\circ d_0^*(U)=d_1^*(U)\)
  for~\(U\) to be a representation implies
  \begin{equation}
    \label{eq:rep-relations}
    U_a\circ U_b=U_{a b}\qquad
    \text{for all }a,b\in S\text{ with } a^*a=bb^*.
  \end{equation}
  Fix \(a,b\in S\) with \(a^*a=bb^*\) as above.  Let \(t\defeq
  a\times_{\s,\rg}b\subseteq G^2\) be the set of all pairs
  \((g,h)\in G^1\times G^1\) with \(g\in a\), \(h\in b\) and
  \(\s(g)=\rg(h)\).  This is an open subset of~\(G^2\), and all
  three vertex maps \(v_i\colon G^2\to G^0\), \(i=0,1,2\), are
  injective on~\(t\) because \(a\) and~\(b\) are bisections
  of~\(G\).  We call subsets of~\(G^2\) with these properties
  \emph{trisections}.  Since \(a^*a=bb^*\), we have
  \[
  v_0(t)=\rg(a)=\rg(ab),\qquad
  v_1(t)=\s(a)=\rg(b),\qquad
  v_2(t)=\s(b)=\s(ab).
  \]
  The face maps \(d_i\colon G^2\to G^1\) for \(i=0,1,2\), are also
  injective on~\(t\), and
  \[
  d_0(t)=a,\qquad
  d_1(t)=ab,\qquad
  d_2(t)=b.
  \]
  The pullback \(d_i^*(U)\colon
  \Lt^2(v_j)\otimes_{\Cont_0(G^0)}\Hilm[F]\congto
  \Lt^2(v_k)\otimes_{\Cont_0(G^0)}\Hilm[F]\) for \(i\in \{0,1,2\}\) and
  appropriate \(j,k\in \{0,1,2\}\) depending on~\(i\) is defined as
  the map
  \[
  \Id\otimes U\colon
  \Lt^2(d_i)\otimes_{\Cont_0(G^1)}\Lt^2(\s)\otimes_{\Cont_0(G^0)}\Hilm[F]\congto
  \Lt^2(d_i)\otimes_{\Cont_0(G^1)}\Lt^2(\rg)\otimes_{\Cont_0(G^0)}\Hilm[F]
  \]
  composed with the canonical isomorphisms
  \[
  \Lt^2(d_i)\otimes_{\Cont_0(G^1)}\Lt^2(\s)\cong \Lt^2(v_j),\qquad
  \Lt^2(d_i)\otimes_{\Cont_0(G^1)}\Lt^2(\rg)\cong \Lt^2(v_k).
  \]
  Therefore, the restriction of~\(d_i^*(U)\) to \(t\subseteq G^2\)
  corresponds to the restriction of \(\Id\otimes U\) to~\(t\).  As
  before, this gives an isomorphism
  \begin{multline*}
    \Cont_0(d_i(t))\cdot\Lt^2(\s)\otimes_{\Cont_0(G^0)}\Hilm[F]
    \cong \Cont_0(t)\cdot\Lt^2(d_i)\otimes_{\Cont_0(G^1)}\Lt^2(\s)
    \otimes_{\Cont_0(G^0)}\Hilm[F]
    \\\congto \Cont_0(t)\cdot
    \Lt^2(d_i)\otimes_{\Cont_0(G^1)}\Lt^2(\rg)\otimes_{\Cont_0(G^0)}\Hilm[F]
    \cong \Cont_0(d_i(t))\cdot \Lt^2(\rg)\otimes_{\Cont_0(G^0)}\Hilm[F].
  \end{multline*}
  It coincides with the restriction of~\(U\)
  to~\(d_i(t)\).
  The equality \(d_2^*(U)\circ d_0^*(U)=d_1^*(U)\)
  implies an equality for all these restrictions; in particular, it
  implies \(d_2^*(U)|_t\circ d_0^*(U)|_t = d_1^*(U)|_t\).
  This is equivalent to \(U_a\circ U_b = U_{a b}\)
  by the above identifications.  Thus~\(\varphi\)
  and the maps~\(U_a\)
  for \(a\in S\)
  form a representation of \((X,S,\vartheta)\)
  if \((\varphi,U)\) is a representation of~\(G\).

  Next we reverse this construction, showing that any covariant
  representation \(\varphi,(U_a)_{a\in S}\) comes from a unique
  representation~\((\varphi,U)\).  We continue to identify elements
  of~\(S\) with the corresponding bisections in~\(G\), which are
  open subsets of~\(G\).  First, we claim that the restrictions of
  \(U_a\) and~\(U_b\) to \(\Hilm[F]_{\s(a\cap b)}\) are equal for all
  \(a,b\in S\).  Definition~\ref{def:cov-rep}.\ref{en:cov-rep1}
  implies that \(U_a\) and~\(U_b\) coincide on \(\Hilm[F]_{\s(c)}\) for
  each \(c\in S\) with \(c \le a,b\).  Since~\(S\) is a wide inverse
  semigroup in~\(\Bis(G)\), \(a\cap b\) is the union of the
  bisections \(c\in S\) with \(c\le a,b\).  Thus \(\Hilm[F]_{\s(a\cap b)}
  = \sum_{c\le a,b} \Hilm[F]_{\s(c)}\), and we get the claim.

  Extending a function on \(a\subseteq G^1\)
  by~\(0\) and summing, we map \(\bigoplus_{a\in S} \Contc(a)\) to
  \(\Contc(G^1)\).  A partition of unity argument shows that this
  map to \(\Contc(G^1)\) is surjective.  We define similar maps
  \begin{align*}
    \tau_\s\colon \bigoplus_{a\in S} \Contc(a)\odot \Hilm[F]_{\s(a)} &\to
    \Lt^2(G^1,\s,\tilde\alpha)\otimes_{\Cont_0(G^0)}\Hilm[F],\\
    \tau_\rg\colon \bigoplus_{a\in S} \Contc(a)\odot \Hilm[F]_{\rg(a)} &\to
    \Lt^2(G^1,\rg,\alpha)\otimes_{\Cont_0(G^0)}\Hilm[F].
  \end{align*}
  We claim that both have dense range.  If we replaced
  \(\Hilm[F]_{\s(a)}\) and~\(\Hilm[F]_{\rg(a)}\) by~\(\Hilm[F]\), this would follow from
  the density of \(\Contc(G^1)\odot \Hilm[F]\) in the right hand sides.
  Since \(\Contc(a) = \Contc(a)\cdot\Contc(a)\), we may rewrite the
  image of \(f\otimes \xi\) for \(f\in\Contc(a)\), \(\xi\in\Hilm[F]\) as
  \(f_1\cdot f_2 \otimes \xi \equiv f_1\otimes f_2\cdot\xi\) with
  \(f_1,f_2\in \Contc(a)\) and hence \(f_2\cdot\xi\in \Hilm[F]_{\s(a)}\)
  when we work in
  \(\Lt^2(G^1,\s,\tilde\alpha)\otimes_{\Cont_0(G^0)}\Hilm[F]\) and
  \(f_2\cdot\xi\in \Hilm[F]_{\rg(a)}\) when we work in
  \(\Lt^2(G^1,\rg,\alpha)\otimes_{\Cont_0(G^0)}\Hilm[F]\).  The
  maps~\((U_a)_{a\in S}\) give an isomorphism
  \[
  U\colon \bigoplus_{a\in S} \Contc(a)\odot \Hilm[F]_{\s(a)} \congto
  \bigoplus_{a\in S} \Contc(a)\odot \Hilm[F]_{\rg(a)},\qquad
  (f_a \otimes \xi_a)_{a\in S} \mapsto
  (f_a \otimes U_a(\xi_a))_{a\in S}.
  \]
  We claim that \(\braket{\tau_\rg\circ U(x)}{\tau_\rg\circ U(y)} =
  \braket{\tau_\s(x)}{\tau_\s(y)}\) for all \(x,y\in \bigoplus_{a\in
    S} \Contc(a)\odot \Hilm[F]_{\s(a)}\).  Hence there is a unique unitary
  operator
  \[
  \bar{U}\colon \Lt^2(G^1,\s,\tilde\alpha)\otimes_{\Cont_0(G^0)}\Hilm[F] \congto
  \Lt^2(G^1,\rg,\alpha)\otimes_{\Cont_0(G^0)}\Hilm[F]
  \]
  with \(\bar{U}\circ \tau_\s = \tau_\rg\circ U\).  It suffices to
  prove the claim if \(x=f_1\otimes \xi_1\) and \(y=f_2\otimes
  \xi_2\) with \(f_1\in \Contc(a)\), \(\xi_1\in\Hilm[F]_{\s(a)}\),
  \(f_2\in \Contc(b)\), \(\xi_2\in\Hilm[F]_{\s(b)}\) for some \(a,b\in
  S\).  So we must prove that
  \[
  \braket{\xi_1}{\braket{f_1}{f_2}_{\s} \xi_2}
  = \braket{U_a(\xi_1)}{\braket{f_1}{f_2}_{\rg} U_b(\xi_2)}.
  \]
  The function \(\braket{f_1}{f_2}_{\s}\) is simply the function
  \(\s(x)\mapsto f_1(x)\cdot \conj{f_2(x)}\) for \(x\in a\cap b\),
  extended by~\(0\) outside \(\s(a\cap b)\), and
  \(\braket{f_1}{f_2}_{\rg} = \braket{f_1}{f_2}_{\s} \circ
  \vartheta_a^*\).  We may rewrite \(\braket{f_1}{f_2}_\s =
  \conj{f_3}\cdot f_4\) with \(f_3,f_4\in \Contc(\s(a\cap b))\).
  Hence \(\braket{\xi_1}{\braket{f_1}{f_2}_{\s} \xi_2} = \braket{f_3
    \xi_1}{f_4 \xi_2}\) and
  \begin{multline*}
    \braket{U_a(\xi_1)}{\braket{f_1}{f_2}_{\rg} U_b(\xi_2)}
    = \braket{U_a(\xi_1)}
    {(\braket{f_1}{f_2}_{\s} \circ \vartheta_b^*) U_b(\xi_2)}
    = \braket{(f_3\circ\vartheta_a^*) U_a(\xi_1)}
    {(f_4\circ\vartheta_b^*) U_b(\xi_2)}
    \\= \braket{U_a(f_3 \xi_1)}{U_b(f_4 \xi_2)}
    = \braket{U_a(f_3 \xi_1)}{U_a(f_4 \xi_2)}
    = \braket{f_3 \xi_1}{f_4 \xi_2}.
  \end{multline*}
  Here we use that \(\vartheta_a\) and~\(\vartheta_b\) agree on \(\s(a\cap
  b)\) and hence on the supports of \(f_3,f_4\), and that \(U_a\)
  and~\(U_b\) agree on \(\Hilm[F]_{\s(a\cap b)}\) and are unitary.  The
  computation above proves our claim that \(\braket{\tau_\rg\circ
    U(x)}{\tau_\rg\circ U(y)} = \braket{\tau_\s(x)}{\tau_\s(y)}\)
  for all \(x,y\in \bigoplus_{a\in S} \Contc(a)\odot \Hilm[F]_{\s(a)}\).
  This finishes the construction of the unitary~\(\bar{U}\).

  The unitary~\(\bar{U}\) acts
  by~\(U_a\) on \(\Contc(a)\cdot \Lt^2(G^1,\s,\tilde\alpha)
  \otimes_{\Cont_0(X)} \Hilm[F]\) because the latter is the closure of the
  \(\tau_{\s}\)\nb-image of \(\Contc(a)\otimes \Hilm[F]_{\s(a)}\).  The
  covariance condition~\ref{en:cov-rep4} in
  Definition~\ref{def:cov-rep} for the unitaries~\(U_a\) says
  that~\(\bar{U}\) intertwines the left actions of \(\Cont_0(G^1)\),
  that is, it is an isomorphism of correspondences.  Since the
  trisections described above cover~\(G^2\), the equality
  \(d_2^*(U)\circ d_0^*(U)=d_1^*(U)\) follows
  from~\ref{eq:rep-relations} by reversing the computation above.
  Thus~\(\bar{U}\) is a representation of~\(G\).  This finishes the
  proof of the bijection between representations of~\(G\)
  and covariant representations of \((X,S,\vartheta)\).

  Finally, we check that our bijection between representations
  of~\(G\) and covariant representations of
  \((\Cont_0(X),S,\vartheta)\) also satisfies the naturality properties
  in Theorem~\ref{the:universal_groupoid_Haus}.
  Let \(V\colon \Hilm[F]\hookrightarrow\F'\) be a Hilbert module isometry.
  Let \(\Hilm[F]\) and~\(\F'\) carry representations \((\varphi,U)\) and
  \((\varphi',U')\) of~\(G\).  If~\(V\) intertwines these
  representations, then it maps \(\Hilm[F]_{a^* a}\) into~\(\F'_{a^* a}\)
  for all \(a\in S\), and the restricted isometries \(\Hilm[F]_{a^* a}
  \hookrightarrow \F'_{a^* a}\) and \(\Hilm[F]_{a a^*} \hookrightarrow
  \F'_{a a^*}\) intertwine the unitaries \(U_a\colon \Hilm[F]_{a^* a}
  \congto \Hilm[F]_{a a^*}\) and \(U'_a\colon \F'_{a^* a} \to \F'_{a
    a^*}\).  Thus~\(V\) intertwines the covariant
  representations of~\((X,S,\vartheta)\) associated to~\((\varphi,U)\).
  Conversely, if~\(V\) intertwines \(\varphi\) and~\((U_a)_{a\in
    S}\) and the corresponding family \(\varphi'\)
  and~\((U'_a)_{a\in S}\), then it must intertwine \(U\) and~\(U'\)
  because we may reconstruct~\(U\) from~\((U_a)_{a\in S}\) as above.
  Thus our bijection has the first naturality property in
  Theorem~\ref{the:universal_groupoid_Haus}.  The second naturality
  property is also routine to check.  Since the bijection between
  representations of \(\Cst(G)\) and \(S\ltimes_\vartheta\Cont_0(X)\)
  has these two naturality properties, it is induced by an
  isomorphism \(\Cst(G)\cong S\ltimes_\vartheta\Cont_0(X)\).
\end{proof}

\end{document}